\documentclass[12pt,reqno]{amsart}
\usepackage{amsmath, amsthm, amssymb,stmaryrd}
\usepackage{mathtools}
\usepackage[OT2,T1]{fontenc}

\usepackage[margin=1in]{geometry}
\usepackage{parskip}
\usepackage[numbers]{natbib}
\usepackage{cleveref}

\usepackage{graphicx}
\usepackage{tikz}
\usepackage{wrapfig}
\usetikzlibrary{matrix, arrows}
\usetikzlibrary{cd}

\makeatletter
\def\thm@space@setup{%
  \thm@preskip=0.5em\thm@postskip=\thm@preskip%
}
\makeatother
\newtheorem*{theorem}{Theorem}
\theoremstyle{plain}
\newtheorem{thm}{Theorem}[section]
\newtheorem{prop}[thm]{Proposition}
\newtheorem{lem}[thm]{Lemma}
\newtheorem{cor}[thm]{Corollary}
\newtheorem{prop/alg}[thm]{Proposition/Algorithm}

\theoremstyle{definition}
\newtheorem{defn}[thm]{Definition}

\newtheorem{eg}[thm]{Example}

\theoremstyle{remark}
\newtheorem{rmk}[thm]{Remark}
\crefformat{footnote}{#2\footnotemark[#1]#3}

\def\A{\mathbb{A}}

\def\Z{\mathbb{Z}}
\def\R{\mathbb{R}}
\def\Q{\mathbb{Q}}
\def\G{\mathbb{G}}
\def\C{\mathbb{C}}
\def\F{\mathbb{F}}

\def\O{\mathcal{O}}
\def\cG{\mathcal{G}}

\def\P{\mathcal{P}}
\def\W{\mathcal{W}}

\def\D{\mathcal{D}}

\def\n{\mathfrak{n}}

\def\gl{\mathfrak{gl}}
\def\so{\mathfrak{so}}
\def\t{\mathfrak{t}}
\def\b{\mathfrak{b}}
\def\g{\mathfrak{g}}
\def\m{\mathfrak{m}}
\def\z{\mathfrak{z}}
\def\l{\mathfrak{l}}
\def\k{\mathfrak{k}}
\def\h{\mathfrak{h}}

\DeclareSymbolFont{cyrletters}{OT2}{wncyr}{m}{n}
\DeclareMathSymbol{\Sha}{\mathalpha}{cyrletters}{"58}

\def\ol{\overline}
\def\ssm{\smallsetminus}
\def\wh{\widehat}
\def\bpm{\begin{pmatrix}}
\def\epm{\end{pmatrix}}
\def\lb{\llbracket}
\def\rb{\rrbracket}

\DeclareMathOperator{\Ram}{Ram}
\DeclareMathOperator{\unr}{unr}

\DeclareMathOperator{\GL}{GL}
\DeclareMathOperator{\Gal}{Gal}
\DeclareMathOperator{\ab}{ab}
\DeclareMathOperator{\Spec}{Spec}
\DeclareMathOperator{\HT}{HT}

\DeclareMathOperator{\Lift}{Lift}
\DeclareMathOperator{\Def}{Def}
\DeclareMathOperator{\Hom}{Hom}
\DeclareMathOperator{\image}{image}
\DeclareMathOperator{\fr}{fr}

\DeclareMathOperator{\Ad}{Ad}
\DeclareMathOperator{\simpc}{sc}
\DeclareMathOperator{\res}{res}
\DeclareMathOperator{\ord}{ord}
\DeclareMathOperator{\Sym}{Sym}

\DeclareMathOperator{\SL}{SL}
\DeclareMathOperator{\SO}{SO}
\DeclareMathOperator{\Spf}{Spf}
\DeclareMathOperator{\pr}{par}
\DeclareMathOperator{\rank}{rank}
\DeclareMathOperator{\rig}{rig}
\DeclareMathOperator{\Rep}{Rep}
\DeclareMathOperator{\Vc}{Vec}
\DeclareMathOperator{\Ext}{Ext}
\DeclareMathOperator{\coker}{coker}
\DeclareMathOperator{\der}{der}

\title[A sparsity of automorphic points]{Galois deformation spaces with a sparsity of automorphic points}
\author{Kevin Childers}

\begin{document}
\maketitle
\begin{abstract}
    Let $k/\F_p$ denote a finite field.  For any split connected reductive group $G/W(k)$ and certain CM number fields $F$, we deform certain Galois representations $\ol\rho:\Gal(\ol F/F) \to G(k)$ to continuous families $X_{\ol\rho}$ of Galois representations $\Gamma_F \to G(\ol\Q_p)$ lifting $\ol\rho$ such that the space of points of $X_{\ol\rho}$ which are geometric (in the sense of the Fontaine-Mazur conjecture) with parallel Hodge-Tate weights has positive codimension in $X_{\ol\rho}$.
    Thus the set of points in $X_{\ol\rho}$ which could (conjecturally) be associated to automorphic forms is sparse.
    This generalizes a result of Calegari and Mazur for $F/\Q$ quadratic imaginary and $G = \GL_2$.
    The sparsity of automorphic points for $F$ a CM field contrasts with the situation when $F$ is a totally real field, where automorphic points are often provably dense.
\end{abstract}
\section{Introduction}\label{intro}

\subsection{Modular Galois representations} Let $\Gamma_\Q$ denote the absolute Galois group of the rational numbers, and $k$ a finite field of characteristic $p$.
Let $c\in\Gamma_\Q$ denote a choice of complex conjugation.
Serre's modularity conjecture \cite{serre}, now a theorem of Khare and Wintenberger proved in \cite{kw1} and \cite{kw2}, states that an absolutely irreducible continuous representation
\[ \ol\rho:\Gamma_\Q \to \GL_2(k) \]
has characteristic 0 lifts which arise from modular forms, provided that $\det\ol\rho(c) = -1$.
We say that a $\ol\rho$ satisfying $\det\ol\rho(c) = -1$ is \textbf{odd}.

In fact, if $\ol\rho$ as above is odd, then \textit{many} characteristic 0 lifts of $\ol\rho$ are modular.
For example, \cite{gm} and \cite{bockle} prove that modular points are dense in the rigid generic fiber of the universal deformation ring.

The situation is quite different if we replace $\Q$ with an imaginary quadratic extension $F/\Q$, where ``oddness'' doesn't even make sense.
On the automorphic side, there are fewer cohomological cuspidal automorphic forms.
One example of this is the vanishing of cuspidal cohomology in non-parallel weight (see e.g.~\cite{harder}, or more generally \cite{bw}).
% For another example, compare the dimension growth of classical spaces of cusp forms as the level increases with \cite[Thm.~1.1]{ce}.
For another example, compare the results of \cite[Thm.~1.1]{ce} with the dimension growth of classical spaces of cusp forms as the level increases

For an example on the Galois side, see Theorem 1.4 of \cite{even1}.
As another example (more relevant to this paper), \cite[\S7]{cm} produces families of characteristic 0 deformations of representations $\ol\rho:\Gamma_F \to \GL_2(k)$ where the Hodge-Tate weights are allowed to vary, and for which very few points have parallel Hodge-Tate weights.
% Note that ``oddness'' doesn't even make sense for $\ol\rho$ over a quadratic imaginary field .
The object of this paper is to conduct a similar study as \cite[\S7]{cm} for representations into more general reductive groups $G$, and for more general CM number fields $F$.

\begin{rmk}
    The results of \cite{BLGGT}, etc.~produce geometric (potentially automorphic) lifts over CM fields, and in fact \cite{allen} even proves some density results for automorphic points in universal polarized deformation rings.
    However, recall that giving a polarized representation over $F$ a CM field is equivalent to giving an odd representation over the maximal totally real subfield of $F^+$, valued in the group $\cG_n$ defined by Clozel, Harris, and Taylor in \cite{CHT}.
\end{rmk}

\subsection{Oddness in general} 

For producing automorphic (or geometric) lifts in characteristic zero, the ``oddness'' hypothesis appears to be essential.\footnote{\cite{cg} provides non-odd modularity lifting theorems, but does so considering torsion classes.}
Oddness takes the following form for higher rank groups.

Suppose that $G$ is a split connected reductive $W(k)$-group scheme, $T$ is a split maximal torus, $B$ is a fixed Borel containing $T$, and $N$ is the unipotent radical of $B$.
Corresponding Lie algebras are denoted by corresponding lower case Fraktur letters.
Let $G^{\der}$ denote the derived group, $\g^0$ the Lie algebra of $G^{\der}$, and $\h^0 = \h \cap \g^0$ for any Lie subalgebra $\h$ of $\g$.
Suppose $F$ is a number field, and we have a continuous homomorphism
\[ \ol\rho:\Gamma_F \to G(k). \]
As explained in the introduction of \cite{CHT} (also see \S 4.5 of \cite{pat}), the classical Taylor-Wiles method requires
\[ [F:\Q]\dim \n = \sum_{v\mid\infty}h^0(\Gamma_{F_v},\g^0). \]
Here, $h^0$ means the dimension of the corresponding $H^0$, and the set of $k$-points of $\g^0$ is given a $\Gamma_F$-module structure via $\ol\rho$ and the adjoint representation.
If $v\mid\infty$ is a complex place, then 
\[ h^0(\Gamma_{F_v},\g^0) = \dim\g^0  = 2\dim\n + \dim\t^0. \]
If $v\mid\infty$ is a real place, then 
\[ h^0(\Gamma_{F_v},\g^0) \geq \dim\n. \]
So we have
\[ \sum_{v\mid\infty}h^0(\Gamma_{F_v},\g^0) \geq [F:\Q]\dim\n + C\dim\t^0, \]
where $C$ is the number of complex places of $F$.
We see that equality can only be achieved if
\begin{enumerate}
    \item $C = 0$, i.e. $F$ is totally real, and
    \item for each $v\mid\infty$, $h^0(\Gamma_{F_v},\g^0) = \dim\n$, i.e. $(\g^0)^{c_v} = \dim\n$.
    In this case, we say that $\ol\rho(c_v)$ is an odd involution, and that $\ol\rho$ is odd.
    Notice that when $G = \GL_2$, this is equivalent to requiring $\det\ol\rho(c_v) = -1$.
\end{enumerate}

\subsection{Non-odd cases}

In this paper we will focus our attention on Galois representations over CM fields, but there are certainly other ways to get non-odd Galois representations, which are worth mentioning.
For example, $\ol\rho:\Gamma_{\Q} \to \GL_2(k)$ could be \textit{even}, in the sense that $\det\ol\rho(c)) = 1$.
Theorem 1.4 of \cite{even2} shows that such a $\ol\rho$ can have no geometric (and thus no modular) characteristic zero deformations.

Another interesting example can be found in \cite{fkp2}, where $\ol\rho:\Gamma_\Q \to \SO_{2n}(k)$.
Then $\ol\rho$ can \textit{never} be odd, but if $\ol\rho$ is ``as odd as possible'' in the sense that $\so_{2n}^{\Ad\ol\rho(c)}$ is as small as possible, then $\ol\rho\oplus 1:\Gamma_\Q \to \SO_{2n+1}(k)$ can be odd.
In theory, one could then imagine constructing a deformation space $X_{\ol\rho}$ for $\ol\rho$ which is a subvariety of a deformation space $X_{\ol\rho\oplus 1}$ for $\ol\rho\oplus 1$, such that the automorphic points of $X_{\ol\rho}$ are sparse, but the automorphic points of $X_{\ol\rho\oplus 1}$ are dense.

\subsection{The method of Ramakrishna} 

In the case $F = \Q$ and $G = \GL_2$, the numerology has been exploited by the methods of Ramakrishna in \cite{r-lift} and \cite{r-def} to produce characteristic 0 lifts of $\ol\rho$ which are not necessarily automorphic, but which are \textit{geometric}, in the sense of the Fontaine-Mazur conjecture (i.e.~unramified at all but finitely many places, and de Rham at places dividing $p$).
An ``axiomatized'' version of Ramakrishna's method can be found in \cite{taylor-icos}, and this axiomatized version was extended in \cite{pat} and \cite{pat2} to produce geometric lifts for any totally real field $F$ and reductive group $G$.
Even more recently, the forthcoming paper \cite{fkp} has vastly improved Ramakrishna's method in the odd case, where essentially the only hypotheses are oddness, irreducibility, and the existence of local lifts.

In the odd case, the more classical Ramakrishna method produces a space of geometric deformations which is represented by a universal deformation ring which is equal to a power series ring over $\O = W(k)$ (the Witt vectors of $k$) in a number of variables equal to
\[ [F:\Q]\dim\n - \sum_{v\mid\infty}h^0(\Gamma_{F_v},\g^0), \]
which comes out to be 0.
In particular, the universal deformation is a geometric lift of $\ol\rho$.

\subsection{Numerology for CM fields} Suppose that $F$ is a CM field (by which I will mean \textit{nonreal} CM field).
Then we have
\[ \sum_{v\mid\infty}h^0(\Gamma_{F_v},\g^0) = [F:\Q]\dim\n + \frac{[F:\Q]}2\dim\t^0. \]
The numerology breaks down, since we now heuristically expect a universal deformation ring which is a power series ring over $\O$ in
\[ [F:\Q]\dim\n - \sum_{v\mid\infty}h^0(\Gamma_{F_v},\g^0) = -\frac{[F:\Q]}2\dim\t^0 \]
variables.

By relinquishing some control over the Hodge-Tate weights, we will show that we can salvage the method to produce a universal deformation ring $R$ which is a power series ring over $\O$ in
\[ [F:\Q]\dim\b^0 - \sum_{v\mid\infty}h^0(\Gamma_{F_v},\g^0) = +\frac{[F:\Q]}2\dim\t^0 \]
variables.
Then, one can investigate which $\ol\Q_p$-points of the universal deformation ring $R$ could be geometric, automorphic, etc.
Due to the vanishing of cuspidal cohomology in non-parallel weights (which will be explained more in \Cref{symmetry}), automorphic points are expected to have parallel Hodge-Tate weights (which will be defined in \Cref{weights}), and to be geometric in the sense of the Fontaine-Mazur conjecture (see \cite{fm}).
It will be shown that the universal deformation space can be constructed such that the set of $\ol\Q_p$-points which correspond to geometric representations with parallel weights lives in a subvariety of positive codimension (at least in the rigid analytic topology), thus $R$ has (conjecturally) a sparsity of automorphic points.
For examples of conjectures relating Galois representations to automorphic forms, see \cite{bg}.
A sample version of our main theorem is:

\begin{thm}\label{mainthm}
    Let $p$ be a prime number, $k$ a finite field of characteristic $p$, $G/W(k)$ a split connected reductive group, $F$ a Galois CM number field satisfying $[F(\mu_p):F] = p-1$ with maximal totally real subextension $F^+$ for which every prime $v\mid p$ of $F^+$ splits in $F$, and $\Sigma$ a finite set of primes of $F$ containing the primes above $p$.
    Assume that $p \gg 0$ with respect to $G$.
    Assume $\ol\rho:\Gamma_{F,\Sigma}\to G(k)$ is a continuous representation, satisfying the following conditions.
    \begin{enumerate}
        \item $\ol\rho(\Gamma_{F,\Sigma}) \supseteq G^{\der}(k)$.
        \item For all places $v \in \Sigma$ not dividing $p\cdot\infty$, $\ol\rho|_{\Gamma_{F_v}}$ satisfies a liftable local deformation condition $\P_v$ with tangent space of dimension $h^0(\Gamma_{F_v},\g^0)$.
        \item For all places $v\mid p$, $\ol\rho|_{\Gamma_{F_v}}$ is nearly ordinary and non-split (see \Cref{d-non-split}), such that $\alpha\circ\ol\rho$ is not trivial or isomorphic to the cyclotomic character for each simple root $\alpha$ of $G$.
    \end{enumerate}
    Then, letting $r = \frac{[F:\Q]}2\dim\t^0$, there exists a continuous representation
    \[ \rho : \Gamma_F \to G(W(k)\lb X_1,\ldots,X_r\rb) \] lifting $\ol\rho$, and such that $\rho$ is almost everywhere unramified, and nearly ordinary at all $v\mid p$.
    Further, the set of $\ol\Q_p$-points of $\rho$ which are geometric and have parallel Hodge-Tate weights has positive codimension in $\D^r$, the rigid open unit $n$-ball.
\end{thm}

See \Cref{final-theorem} for slight improvements and for the proof.
This generalizes the results of \S7 of \cite{cm}, which provides a similar result for $F$ a quadratic imaginary extension of $\Q$ and $G = \GL_2$.
The claim of Theorem 7.1 in \cite{cm} is that \textit{finitely} many specializations of $\rho$ have parallel weight, but as explained by David Loeffler in \cite{loeffler}, \cite{cm} doesn't actually prove this.
Finiteness would follow if $\rho$'s parallel weight specializations lived in positive codimension in the \textit{Zariski topology}.  
However, the result does hold in the \textit{rigid analytic topology}, and our proof that the geometric parallel weight specializations live in a positive codimension subvariety will also only be valid in the rigid analytic topology.
We will mention some modifications to the allowable points which do give results in the Zariski topology.

Theorem 8.9 of \cite{cm} makes a similar finiteness claim (on the automorphic side), which Loeffler \cite{loeffler} explains has a similar problem.
However, Serban has proven finiteness results for Bianchi modular forms in \cite{vlad} which complete the proof of \cite[Thm.~8.9]{cm}, and has some generalizations to higher rank.

Due to \cite{scholze}, \cite{hltt}, and many others, we can associate Galois representations to regular algebraic automorphic forms on $\GL_n(\A_F)$ over CM fields $F$.
In cases where local-global compatibility at $\ell=p$ is known, we even know that such Galois representations are geometric with parallel weight (see e.g.~the condition on the Hodge-Tate weights in \cite[Thm.~2.1.1]{BLGGT}).
On the geometric side of things, all pure homological motives $X$ over $F$ satisfy Hodge-symmetry.
The comparison isomorphisms of $p$-adic Hodge theory then imply that the $p$-adic \'etale cohomology of $X$ is geometric of parallel weight.
Call such a Galois representation \textit{motivic}.

\begin{cor}Let $\rho$ satisfy the conclusions of \Cref{mainthm}.
Then the following hold.
    \begin{itemize}
        \item The subset of $\D^r$ whose points correspond to motivic Galois representations has positive codimension.
        \item If $G = \GL_n$, then the subset of $\D^r$ whose points correspond to Galois representations which are known to be associated to automorphic forms has positive codimension.
    \end{itemize}
\end{cor}

\Cref{s-example} provides examples where the hypotheses of \Cref{mainthm} are satisfied.
The following is proven as \Cref{c-example}.

\begin{thm}\label{t-example}
    For any split connected reductive group $G/W(k)$ for which $k$ is a finite field of charicteristic $p \gg 0$, there exists a Galois CM field $F$, a set of primes $\Sigma$, and a Galois representation $\ol\rho:\Gamma_{F,\Sigma} \to G(k)$ satisfying the conditions of \Cref{mainthm}.
\end{thm}

\subsection{Organization} Section \ref{notation} will set up notation used throughout the remainder of the paper.
Section \ref{deftheory} briefly reviews the necessary theory of Galois deformations, largely following \S\S3-4 of \cite{pat}.
We will also introduce the ``nearly ordinary'' local condition that allows us to adjust the numerology above.
Section \ref{rammethod} describes the necessary alterations to the Ramakrishna method for constructing a universal deformation ring of the desired size.
In \Cref{gl2-weights} we prove the main result of Section 7 in \cite{cm}, addressing some subtleties that do not appear to be addressed there.
\Cref{symmetry} will briefly review the motivic and automorphic picture, to motivate the necessity of ``parallel Hodge-Tate weights.''
\Cref{weights} will generalize the arguments of \Cref{gl2-weights} to more general fields $F$ and groups $G$ to reduce the proof of \Cref{mainthm} to a calculation in Galois cohomology.
Section \ref{sparse} will then provide the necessary Galois cohomology computations to complete the proof of \Cref{mainthm}.
Finally, in Section we prove \Cref{t-example}.

\subsection{Acknowledgements} I am deeply grateful for the encouragement and support of Stefan Patrikis, and the many helpful conversations we had with regards to this paper.
I also thank Frank Calegari and Shiang Tang for their feedback and suggestions.
Finally, I thank my wife Tessa and my children Peter and Amelia for their love, patience, support, and excitement.
\section{Basic notation}\label{notation}

Most of the notation will match that of \cite{pat}.
One notable exception is that this paper will swap the roles of $p$ and $\ell$, which better matches \cite{cm}.

Throughout, $p$ is a fixed odd prime number and $k$ a fixed finite field of characteristic $p$.
Denote the ring of Witt vectors of $k$ by $\O$, and the field of fractions of $\O$ by $E$.

Let $G$ be a split connected reductive $\O$-group scheme.
We will denote the Lie algebra of $G$ by $\g$, and similarly denote Lie algebras of other algebraic groups by using the lower case Fraktur form of the letter.
We write $G^{\der}$ for the derived group of $G$, $\g^0 = [\g,\g]$, and $\h^0 = \h \cap \g^0$ for any Lie subalgebra $\h$ of $\g$.

We will assume that $p$ is a ``very good prime'' for $G$.
For a full definition of this concept see \S1.14 of \cite{carter}. 
The relevant consequences are summarized in \S3 of \cite{pat}.
Most importantly, $\g = \z(\g)\oplus\g^0$, where $\z(\g)$ is the center of $\g$, $\g^0$ is an irreducible representation, and there is a non-degenerate $G$-invariant form $\g^0\times\g^0\to k$, which allows us to fix an identification of $\g^0$ with its linear dual.

$F$ will usually denote a number field unless otherwise noted.
In any case, $\Gamma_F = \Gal(\ol F/F)$ for some fixed algebraic closure $\ol F$ of $F$.
When $F$ is a number field and $\Sigma$ is any set of places of $\Sigma$, $F_\Sigma$ denotes the maximal extension of $F$ inside $\ol F$ which is unramified outside of $\Sigma$, and $\Gamma_{F,\Sigma} = \Gal(F_\Sigma/F)$.
We also fix for each place $v$ of $F$ embeddings $\ol F \to \ol F_v$, and thus embeddings $\Gamma_{F_v} \to \Gamma_F$.

All representations will be assumed continuous.
For a residual representation
\[ \ol\rho : \Gamma_F \to G(k), \]
When we write $\g$ (or any subalgebra) as a $\Gamma_F$-modlue, we mean the $k$-points of $\g$, with $\Gamma_F$-action given by the composition of $\ol\rho$ and the adjoint representation.

Let $\kappa$ denote the $p$-adic cyclotomic character ( we normalize Hodge-Tate weights so that all weights of $\kappa$ are $+1$).
We write $h^i(\cdot)$ to denote the dimensions of cohomology groups $H^i(\cdot)$.
\section{Deformation theory}\label{deftheory}

\subsection{Introduction} Following \cite{pat}, we recall some facts about deformation theory for a Galois representation valued in a general reductive group, then explain the local conditions used in this paper.
Whenever more details or references are desired, see \S\S3-4 of \cite{pat}.

\subsection{Lifts and deformations}
Suppose for now that $\Pi$ is a profinite group satisfying Mazur's $p$-finiteness condition (see \cite{mazur}), and $\ol r:\Pi \to G(k)$.
Let $\mathcal C_\O$ denote the category whose objects are complete local Noetherian $\O$-algebras with residue field $k$, and whose morphisms are local $\O$-algebra maps.
The functor 
\[ \Lift_{\ol r}:\mathcal C_\O \to \text{\textbf{Sets}} \]
defined by
\[ \Lift_{\ol r}(R) = \{\text{representations $\Gamma_F\to G(R)$ lifting $\ol r$} \} \]
is representable by a ring which we write $R^\Box_{\ol r}$.
We say that $r,r' \in \Lift_{\ol r}(R)$ are \textbf{strictly equivalent} if they are conjugate by an element of the kernel of the reduction map
\[\wh G(R) := \ker(G(R) \to G(k)). \]
Note that $\wh G$ is represented by a smooth group scheme over $\O$.
Define a functor
\[ \Def_{\ol r}:\mathcal C_\O \to \text{\textbf{Sets}} \]
by
\[ \Def_{\ol r}(R) = \Lift_{\ol r}(R)/\sim, \]
where $\sim$ denotes strict equivalence.

\subsection{Deformation conditions}
A \textbf{deformation condition} is a representable subfunctor of $\Lift_{\ol r}$ that is closed under strict equivalence.
The representing ring is always a quotient of $R^\Box_{\ol r}$ by an ideal invariant under $\wh G(R^\Box_{\ol r})$.
If $\Lift_{\ol r}^\P$ is a deformation condition, define $\Def_{\ol r}^\P:\mathcal C_\O \to $\textbf{Sets} by
\[ \Def_{\ol r}^\P(R) = \Lift_{\ol r}^\P(R)/\sim. \]
$\Def_{\ol r}^\P$ is not necessarily representable, but will be in the case that $\g^{\ol r(\Pi)} = \z(\g)$.
When $\Def_{\ol r}^\P$ is representable, write $R_{\ol r}^\P$ for the representing ring.

\begin{eg}
    Define $\ab : G \twoheadrightarrow G/G^{\der} =: S$ so be the abelianization map, and fix a lift $\nu:\Pi \to S(\O)$ of $\ab\circ \ol r$.
    Then 
    \[ \Lift_{\ol r}^\nu(R) := \{r \in \Lift_{\ol r}(R) \mid \ab\circ r = \nu\} \]
    defines a deformation condition, analogous to ``fixing the determinant'' in the $G = \GL_n$ case.
\end{eg}

\subsection{The tangent space} Let $\varpi$ denote a uniformizer of $\mathcal O$, $\m_R$ the maximal ideal of $R_{\ol r}^\Box$, and write $k[\epsilon]$ for ring $k[\epsilon]/(\epsilon^2)$.
There are canonical isomorphisms identifying the tangent space at the closed point of $R^\Box_{\ol r}$
\[ \Hom_k(\m_R/(\m_R^2,\varpi),k) \cong \Lift_{\ol r}(k[\epsilon]) \cong Z^1(\Pi,\g), \]
inducing $\Def_{\ol r}(k[\epsilon]) \cong H^1(\Pi,\g)$.

For a deformation condition $\Lift_{\ol r}^\P$ of $\Lift_{\ol r}$, we can identify subspaces $L^{\Box,\P} \subset Z^1(\Pi,\g)$ and $L^\P \subset H^1(\Pi,\g)$ satisfying $\Lift_{\ol r}^\P(k[\epsilon]) \cong L^{\Box,\P}$, $\Def_{\ol r}^\P(k[\epsilon]) \cong L^\P$, and fitting into an exact sequence
\[ 0 \to H^0(\Pi,\g) \to \g \to L^{\Box,\P} \to L^\P \to 0. \]

\subsection{Liftable deformation problems} A surjection $R \to R/I$ is \textbf{small} if $\m_R I = 0$.
A deformation condition $\Lift_{\ol r}^\P$ is \textbf{liftable} if for any small surjection $R \to R/I$, $\Lift_{\ol r}^\P(R) \to \Lift_{\ol r}^\P(R/I)$ is also surjective.
This condition is equivalent to $R_{\ol r}^{\Box,\P}$ being isomorphic to a power series ring in $\dim_k L^{\Box,\P}$ variables.

\subsection{Local deformation conditions} Suppose that $\Sigma$ is some set, and for each $v \in \Sigma$, $\Pi_v$ is a pro-finite group also satisfying Mazur's $p$-finiteness condition, together with a map $\iota_v:\Pi_v \to \Pi$.
For each $v\in \Sigma$, also fix a deformation condition $\Lift_{\ol r_v}^{\P_v}$ for the representation $\ol r_v = \ol r \circ \iota_v$.
Then we can form a deformation condition for $\ol r$, which we write $\Lift_{\ol r}^\P$ consisting of lifts $r \in \Lift_{\ol r}(R)$ such that $r_v = r\circ \iota_v \in \Lift_{\ol r_v}^{\P_v}$ for all $v\in \Sigma$.
The tangent space $\Def_{\ol r}^\P(k[\epsilon])$ can be identified with a Selmer group
\[ H^1_\P(\Pi,\g) = \ker\left(H^1(\Pi,\g) \to \bigoplus_{v\in \Sigma}H^1(\Pi_v,\g)/L_v\right), \]
where $L_v = L^{\P_v}$ are the local tangent spaces.
One can also (as in \S2.2 of \cite{CHT} or \S3.2 of \cite{pat}) define a group $H^2_\P(\Pi,\g)$ measuring obstructions to liftability, as made precise in the following proposition.
For details, see \S3.2 of \cite{pat}.
% For the original proof in the $\GL_2$ case, see \cite{mazur}.
% For a proof in the generality we state, see \cite{CHT}.

\begin{prop}\label{univsize}
    Assume that $\g^{\ol r(\Pi)} = \z(\g)$, and let $\Lift_{\ol r_v}^{\P_v}$ be a collection of liftable local deformation conditions.
    Then the functor $\Def_{\ol r}^\P$ is representable by a universal deformation ring $R_{\ol r}^\P$ which is isomorphic to a power series ring over $\O$ in $h^1_\P(\Pi,\g)$ variables, modulo an ideal generated by at most $h^2_\P(\Pi,\g)$ elements.
\end{prop}

\begin{rmk}
    Requiring a fixed similitude character $\nu:\Pi \to S(\O)$ amounts to replacing $\g$ with $\g^0$ in the group cohomology.
    See Remark 3.8 in \cite{pat}.
\end{rmk}

\subsection{Deformations of Galois representations}We will now take $\ol r = \ol \rho:\Gamma_F \to G(k)$.
In particular, we can take $\Pi = \Gamma_{F,\Sigma}$ for some finite set of primes $\Sigma$ of $F$ (including all primes at which $\ol r$ is ramified, as well as primes above $p$ and $\infty$), and $\Pi_v = \Gamma_{F_v}$ for each $v\in \Sigma$.
All of these pro-finite groups satisfy the $p$-finiteness condition, see \cite{mazur}.
We make the following assumptions:
\begin{enumerate}
    \item All lifts of $\ol\rho$ have fixed similitude character $\nu:\Gamma_{F,\Sigma}\to S(\O)$.
    \item For each $v\in \Sigma$, $\Lift_{\ol\rho_v}^{\P_v}$ is a fixed liftable local deformation condition, $\P = \{\P_v\}_{v\in\Sigma}$, and $\Lift_{\ol\rho}^\P$ is the associated global deformation condition.
    \item $\g^{\ol\rho(\Gamma_{F,\Sigma})} = \z(\g)$, so that \Cref{univsize} applies.
\end{enumerate}

\subsection{Dual Selmer} Write $L_v$ for the tangent space $\Def_{\ol\rho_v}^{\P_v}(k[\epsilon])$, and write $L_v^\perp \subset H^1(\Gamma_{F_v},\g^0(1))$ for the orthogonal compliment of $L_v$ under the local duality pairing, where here we use our fixed identification $(\g^0)^* \cong \g^0$.
Then just as we formed the Selmer group $H^1_\P(\Gamma_{F,\Sigma},\g^0)$, we can form the corresponding dual Selmer group 
\begin{align*}
    H^1_{\P^\perp}&(\Gamma_{F,\Sigma},\g^0(1)) \\
    &= \ker\left(H^1(\Gamma_{F,\Sigma},\g^0(1)) \to \bigoplus_{v\in\Sigma}H^1(\Gamma_{F_v},\g^0(1))/L_v^\perp\right).
\end{align*}
The group $H^1_{\P^\perp}(\Gamma_{F,\Sigma},\g^0(1))$ has the same $k$-dimension as $H^2_\P(\Gamma_{F,\Sigma},\g^0)$ (see \S3.3 of \cite{pat}).
Wiles's formula gives the following.

\begin{prop}\label{wilescriterion}
    With the above hypotheses,
    \begin{align*}
        &h^1_\P(\Gamma_{F,\Sigma},\g^0) - h^1_{\P^\perp}(\Gamma_{F,\Sigma},\g^0(1)) \\
        &\quad= h^0(\Gamma_F,\g^0) - h^0(\Gamma_F,\g^0(1)) + \sum_{v\in\Sigma}(\dim L_v - h^0(\Gamma_{F_v},\g^0)).
    \end{align*}
\end{prop}

\subsection{Numerology} We will put ourselves in a situation where both $h^0(\Gamma_F,\g^0)$ and $h^0(\Gamma_F,\g^0(1))$ vanish, simplifying Wiles's formula to
\begin{equation}\label{wiles1}
    h^1_\P(\Gamma_{F,\Sigma},\g^0) - h^1_{\P^\perp}(\Gamma_{F,\Sigma},\g^0(1)) = \sum_{v\in\Sigma}(\dim L_v - h^0(\Gamma_{F_v},\g^0)).
\end{equation}
A first thing to notice is that adding a prime $v$ to $\Sigma$ and a liftable local condition $\P_v$ at $v$ which satisfies
\begin{equation}\label{basiccondition}
    \dim L_v = h^0(\Gamma_{F_v},\g^0) 
\end{equation}
will have no effect on \cref{wiles1}.
One important example of this are the primes of \textit{Ramakrishna type,} described in \Cref{rammethod}.
Ramakrishna's method uses the addition of primes of Ramakrishna type to annihilate the dual Selmer group on the left hand side of \cref{wiles1}, while leaving the right hand side unchanged, simplifying \cref{wiles1} to 
\begin{equation}\label{wiles2}
    h^1_\P(\Gamma_{F,\Sigma},\g^0) = \sum_{v\in\Sigma}(\dim L_v - h^0(\Gamma_{F_v},\g^0)).
\end{equation}
Then the universal deformation ring $R_{\ol\rho}^\P$ will be isomorphic to a power series ring over $\O$ in this many variables.
Assume for all $v\in\Sigma$, $v\nmid p\infty$ that we have fixed liftable local conditions $\P_v$ satisfying \cref{basiccondition}.
For each $v\mid\infty$, we are forced to take $L_v = 0$ (assuming $p>2$), and \cref{wiles2} becomes
\begin{equation}\label{wiles3}
    h^1_\P(\Gamma_{F,\Sigma},\g^0) = \sum_{v\mid p}(\dim L_v - h^0(\Gamma_{F_v},\g^0)) - \sum_{v\mid \infty}h^0(\Gamma_{F_v},\g^0).
\end{equation}
Of course, the goal is for this number to be $\geq 0$ if we want to produce characteristic 0 lifts.
% For simplicity of discussion assume that $G$ is semi-simple, i.e.~$\g = \g^0$.
Fix a Borel subgroup $B$ containing a fixed maximal torus $T$, and let $N$ denote the unipotent radical of $B$.
In \cite{pat} an ``ordinary with fixed Hodge-Tate weights'' local condition is imposed at $v\mid p$, which is liftable and satisfies
\[ \dim L_v = h^0(\Gamma_{F_v},\g^0) + [F_v:\Q_p]\dim\n. \]
This condition allows control over Hodge-Tate weights, producing lifts which are geometric (in the sense of the Fontaine-Mazur conjecture), and \cref{wiles3} becomes
\begin{equation}\label{wiles4}
h^1_\P(\Gamma_{F,\Sigma},\g^0) = [F:\Q]\dim\n - \sum_{v\mid \infty}h^0(\Gamma_{F_v},\g^0).
\end{equation}
Attempting to force \cref{wiles4} to be nonnegative brings us back to the conclusion of the oddness discussion in the introduction: $F$ must be totally real, and $\ol\rho$ must map complex conjugation to an odd involution.
Under these assumptions, \cref{wiles4} equals 0, and  $R_{\ol\rho}^\P \cong \O$.

However, recall that we are taking $F$ to be a CM field.
With ordinary local conditions at $v\mid p$, the right hand side of \cref{wiles4} equals
\[ [F:\Q]\dim\n - \frac{[F:\Q]}2\dim \g^0 = -\frac{[F:\Q]}2\dim\t^0, \]
so this strategy fails.

\subsection{Nearly ordinary deformations}\label{ss-nearly-ord-defs} In \S7 of \cite{cm}, this problem is investigated with $G = \GL_2$ and $F/\Q$ quadratic imaginary with $p$ split.
At each prime $v\mid p$, they replace the ``ordinary with fixed Hodge-Tate weights'' condition with a ``nearly ordinary'' condition, where the weights are allowed to vary.
As we will explain below, this has the effect of replacing $\n$ in \cref{wiles4} with $\b$.
This allows us to once again produce characteristic 0 deformations, as we have
\[ [F:\Q]\dim\b^0 - \frac{[F:\Q]}2\dim \g^0 = +\frac{[F:\Q]}2\dim\t^0. \]
The trade off is that this no longer guarantees that the deformations produced are geometric (in fact, we will produce examples where the geometric points conjecturally related to automorphic representations are \textit{sparse} in the deformation space).

We will now explain the nearly ordinary condition.
Compare to \S4.1 of \cite{pat}, which describes the ordinary condition with fixed Hodge-Tate weights in the same manner and generality, and \S2 of \cite{cm} which describes their nearly ordinary condition for $\GL_2$.

\begin{defn}[Nearly ordinary deformations]
    Fix $v\mid p$, and suppose that $\ol\rho_v:\Gamma_{F_v}\to G(k)$ takes values in $B(k)$ for a fixed Borel subgroup $B$ of $G$.
    Define $\Lift^{\ord}_{\ol\rho_v}(R)$ to be the subset of $\Lift_{\ol\rho_v}(R)$ of lifts $\rho$ such that there exists $g \in \wh G(R)$ such that $\,^g\rho(\Gamma_F) \subset B(R)$.
    We call such lifts \textbf{nearly ordinary} and we call $\Lift_{\ol\rho_v}^{\ord}$ the \textbf{nearly ordinary deformation condition}.
\end{defn}
Now we will impose the following condition on $\ol\rho_v$:
\[ (\text{REG})\qquad H^0(\Gamma_{F_v},\g/\b) = 0. \]
\begin{lem}
    Suppose $\ol\rho_v$ satisfies (REG).
    Then $\Lift_{\ol\rho_v}^{\ord}$ is a local deformation condition.
\end{lem}
\begin{proof}
    We need to show that this definition is well defined, and that $\Lift_{\ol\rho_v}^{\ord}$ is a representatble subfuctor.
    Both follow from noting that under assumption (REG), if $\rho(\Gamma_{F_v})$ and $\,^g\rho(\Gamma_{F_v}) \subset B(R)$, then $g \in \wh B(R)$.
    For more details and references, see Lemma 4.2 and its proof in \cite{pat}.
\end{proof}
The tangent space to $\Lift_{\ol\rho}^{\ord}$ is described by the following Lemma.
\begin{lem}
    Assume $\ol\rho$ satisfies (REG) so that the preceding lemma applies.
    Then 
    \[ L^{\ord} = \image\left(H^1(\Gamma_{F_v},\b^0) \to H^1(\Gamma_{F_v},\g^0)\right), \]
    which is an injective map.
    Thus we can identify $L^{\ord} = H^1(\Gamma_{F_v},\b^0)$.
\end{lem}
\begin{proof}
    This can be extracted from the proof of Lemma 4.3 of \cite{pat}, but we will repeat the proof here, since it is short and our lemma looks slightly different.
    Assumption (REG) assures that we have a deformation condition, and that the map in the lemma is injective.
    If $\phi \in Z^1(\Gamma_{F_v},\g^0)$ represents a class in $L^{\ord}$, then there is some $X \in \g$ such that
    \[ \exp(\epsilon X)\exp(\epsilon\phi(g))\ol\rho(g)\exp(-\epsilon X) \in B(k[\epsilon]). \]
    Observe that the cycle
    \[ \phi'(g) = \phi(g) + X - \Ad(\ol\rho(g))X \]
    is cohomologous to $\phi$, and $\phi' \in H^1(\Gamma_{F_v},\b^0)$.
\end{proof}
To continue, we need a couple additional assumptions, one of which is the following:
\[ \text{(REG*)}\qquad H^0(\Gamma_{F_v},\g/\b(1)) = 0. \]
\begin{prop}
    Assume that $\ol\rho$ satisfies (REG) and (REG*), and that $\zeta_p \not\in F_v$.
    Then $\Lift_{\ol\rho_v}^{\ord}$ is liftable, and
    \[ \dim L^{\ord} = [F:\Q_p]\dim\b^0 + h^0(\Gamma_{F_v},\g^0). \]
\end{prop}
\begin{proof}
    See Prop. 4.4 in \cite{pat}.
    Again, we repeat the proof since our deformation condition looks slightly different.
    
    The Killing form induces a nondegenarate pairing $\n\times\g/\b\to k$, and thus a $\Gamma_{F_v}$-equivariant isomorphism $\n^* \cong \g/\b$.
    By local duality and (REG*),
    \[ H^2(\Gamma_{F_v},\n) \cong H^0(\Gamma_{F_v},\g/\b(1))^* = 0. \]
    From the long exact sequence associated to $0 \to \b^0 \to \g^0 \to \g/\b \to 0$ and (REG), we see that
    \[ H^0(\Gamma_{F_v},\b^0) \cong H^0(\Gamma_{F_v},\g^0). \]
    From the long exact sequence associated to $0 \to \n \to \b^0 \to \b^0/\n \to 0$ and the fact that $H^2(\Gamma_{F_v},\n) = 0$ by (REG*) we have
    \[ H^2(\Gamma_{F_v},\b^0) \cong H^2(\Gamma_{F_v},\b^0/\n). \]
    Notice that $\b^0/\n$ is a trivial $\Gamma_{F_v}$-module, so by local duality,
    \[ H^2(\Gamma_{F_v},\b^0/\n) \cong H^0(\Gamma_{F_v},\b^0/\n(1)) = 0. \]
    Here we are using our assumption that $\zeta_p \not\in F$.
    Putting everything together using the local Euler characteristic formula we have
    \begin{align*}
        h^1(\Gamma_{F_v},\b^0) 
        &= h^0(\Gamma_{F_v},\b^0) + h^2(\Gamma_{F_v},\b^0) + [F:\Q_p]\dim \b^0 \\
        &= h^0(\Gamma_{F_v},\g^0) + 0 + [F:\Q_p]\dim\b^0,
    \end{align*}
    as desired.
    
    To complete the proof, we desire liftablilty.
    An obstruction to lifting over a small surjection with kernel $I$ lies in $H^2(\Gamma_{F_v},\b^0)\otimes_k I$, and we have already shown that, under our hypotheses, $H^2(\Gamma_{F_v},\b^0) = 0$.
\end{proof}

\begin{rmk}
    The ordinary condition with fixed Hodge-Tate weights of \cite{pat} is defined in the same way, except by requiring of lifts $\rho$ that
    \[ I_{F_v} \xrightarrow{(^g\rho)|_{I_{F_v}}} B(R) \twoheadrightarrow T(R) \]
    is given by a fixed lift $\chi:I_{F_v} \to T(\O)$ of 
    \[ I_{F_v} \xrightarrow{\ol\rho|_{I_{F_v}}} B(k) \twoheadrightarrow T(k). \]
    Choosing $\chi$ allows us to produce characteristic 0 lifts of $\ol\rho$ which are de Rham at $v\mid p$, and therefore geometric.
    See Lemma 4.8 in \cite{pat} 
\end{rmk}
\section{The method of Ramakrishna}\label{rammethod}

\subsection{Introduction} This section begins by describing the local condition at certain auxiliary primes which allow one to annihilate the dual Selmer group.
Then, following \S5 of \cite{pat}, an axiomatized version of Ramakrishna's method is stated which is suitable for our application.
We will refer to \cite{pat} for most of the proofs, indicating the necessary alterations we need.

\subsection{Ramakrishna primes} Let $\ell \neq p$ be a prime, and $v\mid \ell$ a place of $F$ at which $\ol\rho$ is unramified.
Let $\fr_v$ denote a choice of geometric Frobenius for $F_v$.
Suppose that $\ol\rho_v(\fr_v)$ is a regular semi-simple element lying in $T(k)$ for some maximal torus $T$.
\begin{defn}
    We say the representation $\ol\rho_v$ as above is of \textbf{Ramakrishna type} if there exists a \textit{unique} root $\alpha \in \Phi(G,T)$ such that
    \[ \alpha(\ol\rho_v(\fr_v)) = \ol\kappa(\fr_v) = q_{F_v}^{-1}. \]
    Notice that $\ell\not\equiv 1 \pmod p$ automatically.
\end{defn}
\begin{rmk}
    In \cite{pat} it is not required that $\alpha$ be unique, but with this additional constraint, $H^0(\Gamma_{F_v},\g^0(1))$ is 1-dimensional, which will be necessary for the calculations of \Cref{sparse}.
\end{rmk}
\begin{defn}[Ramakrishna deformations]
    Let $H_\alpha = T\cdot U_\alpha$.
    Define $\Lift^{\Ram}_{\ol\rho_v}(R)$ to be the set of lifts $\rho \in \Lift_{\ol\rho_v}(R)$ which are $\wh G(R)$-conjugate to some $\rho':\Gamma_{F_v} \to H_\alpha(R)$ such that the composition
    \[ \Gamma_{F_v} \xrightarrow{\rho'} H_\alpha(R) \xrightarrow{\Ad}\GL_R(\g_\alpha\otimes R) = R^\times \]
    is equal to $\kappa$.
\end{defn}
The following is Lemma 4.10 in \cite{pat}.
\begin{lem}
    For $\ol\rho_v$ of Ramakrishna type, $\Lift^{\Ram}_{\ol\rho_v}$ is well-defined and yields a liftable deformation condition.
\end{lem}
Let $T_\alpha = \ker(\alpha)^\circ$, and $\t_\alpha$ be the Lie algebra of $T_\alpha$.
If $\l_\alpha$ is the Lie algebra of the 1-dimensional torus generated by $\alpha^\vee$, then $\t_\alpha\oplus\l_\alpha \cong \t^0$.
The tangent space $L^{\Ram}$ is described by the following, which is Lemma 4.11 in \cite{pat}.
\begin{lem}
    Let $W = \t_\alpha\oplus\g_\alpha$.
    Then
    \begin{enumerate}
        \item $L^{\Ram} = \image\left(H^1(\Gamma_{F_v},\ol\rho(W)) \to H^1(\Gamma_{F_v},\g^0)\right)$
        \item $\dim L^{\Ram} = h^0(\Gamma_{F_v},\g^0)$
        \item $L^{\Ram,\perp} = \image\left(H^1(\Gamma_{F_v},\ol\rho(W^\perp)(1)) \to H^1(\Gamma_{F_v},\g^0(1))\right)$
        \item All cocycles in $L^{\Ram,\Box}$ have vanishing $\l_\alpha$-component under the canonical decomposition 
        \[ \g^0 = \t_\alpha \oplus \l_\alpha \bigoplus_\gamma \g_\gamma. \]
        All cocycles in $L^{\Ram,\perp,\Box}$ have vanishing $\g_{-\alpha}$-component.
    \end{enumerate}
\end{lem}

\subsection{Axiomatized Ramakrishna method}
We now use primes of Ramakrishna type to construct a universal deformation ring isomorphic to $\O\lb X_1,\ldots,X_{r}\rb$ with $r = \frac{[F:\Q]}2\dim_k\t^0$.
This section is an adjusted version of \S5 of \cite{pat}.
Suppose that $\ol\rho:\Gamma_{F,\Sigma}\to G(k)$ for a finite set of places $\Sigma$ has infinitesimal centralizer $\z(\g)$.
Fix a similitude character $\nu$ lifting $\ab\circ\ol\rho$.
We enforce this similitude character upon all our deformations, which has the effect of replacing $\g$ with $\g^0$ in the Galois cohomology calculations, and similarly with other subquotients of $\g$.
Let $K = F(\g^0,\mu_p)$.
We now enforce the following assumptions.
\begin{enumerate}
    \item $h^0(\Gamma_F,\g^0) = h^0(\Gamma_F,\g^0(1)) = 0$.
    \item There is a global deformation condition $\P = \{\P_v\}_{v\in\Sigma}$ consisting of liftable deformation conditions for each place $v\in\Sigma$ (all with fixed multiplier character);
    the dimensions of their tangent spaces are
    \[ \dim L_v = \begin{cases} h^0(\Gamma_{F_v},\g^0) & \text{if }v\nmid p\cdot\infty \\ h^0(\Gamma_{F_v},\g^0) + [F_v:\Q_p]\dim\b^0 & \text{if }v\mid p. \end{cases} \]
    \item $F$ is a CM field.
    \item $H^1(\Gal(K/F),\g^0) = H^1(\Gal(K/F),\g^0(1)) = 0$.
    \item Assume that (4) holds.
    For any pair of non-zero Selmer classes $\phi \in H^1_{\P^\perp}(\Gamma_{F,\Sigma},\g^0(1))$ and $\psi \in H^1_{\P}(\Gamma_{F,\Sigma},\g^0)$, we can restrict $\phi$ and $\psi$ to $\Gamma_K$, where they become homomorphisms which are nonzero by (4).
    Letting $K_\phi/K$ and $K_\psi/K$ be their respective fixed fields, we assume that $K_\phi$ and $K_\psi$ are linearly disjoint over $K$.
    \item Consider any $\phi$ and $\psi$ as in the hypothesis of item (5).
    Then there is an element $\sigma \in \Gamma_F$ such that $\ol\rho(\sigma)$ is a regular semi-simple element of $G$, the connected component of whose centralizer we denote $T$, and such that there exists a unique root $\alpha \in \Phi(G,T)$ satisfying
    \begin{enumerate}
        \item $\ol\kappa(\sigma) = \alpha\circ\ol\rho(\sigma)$;
        \item $k[\psi(\Gamma_K)]$ has an element with non-zero $\l_\alpha$-component; and
        \item $k[\phi(\Gamma_K)]$ has an element with non-zero $\g_{-\alpha}$-component.
    \end{enumerate}
\end{enumerate}

\begin{rmk}
    Compare these to the 6 assumptions in \S5 of \cite{pat}.
    The only differences are in (2), (3), and the uniqueness assumption on $\alpha$ in (6).
\end{rmk}

Under the above assumptions, we can now state our analogue of Prop.~5.2 in \cite{pat}.

\begin{prop}\label{ax-ram}
    Under assumptions (1)-(6) above, there exists a finite set of primes $Q$ disjoint from $\Sigma$, such that the universal deformation ring $R_{\ol\rho}^{\{\P_v\}_{v\in\Sigma\cup Q}}$ with $\P_v$ as in (2) for $v\in \Sigma$ and $\P_v$ the Ramakrishna condition for $v \in Q$ is isomorphic to a power series ring over $\O$ in  $\frac{[F:\Q]}2\dim\t^0$ many variables.
\end{prop}

\begin{proof}
    The proof is nearly the same as the proof of Prop.~5.2 in \cite{pat}, but we will repeat most of the proof, since we will need to refer back to it.
    % For the majority of the proof, we refer back to the proof of Prop. 5.2 in \cite{pat}.
    % The only differences are that we are allowing $F$ to be imaginary CM, and compensating by replacing the local condition an $v\mid p$ with a local condition with a larger tangent space.
    % These assumptions are used in the following way.
    Wiles's formula gives us
    \begin{align*}
        &h^1_\P(\Gamma_{F,\Sigma},\g^0) - h^1_{\P^\perp}(\Gamma_{F,\Sigma},\g^0(1)) \\
        &\quad= h^0(\Gamma_F,\g^0) - h^0(\Gamma_F,\g^0(1)) + \sum_{v\in\Sigma}(\dim L_v - h^0(\Gamma_{F_v},\g^0)),
    \end{align*}
    which under assumptions (1), (2), and (3) reduces to
    \begin{align*}
    h^1_\P(\Gamma_{F,\Sigma},\g^0) -& h^1_{\P^\perp}(\Gamma_{F,\Sigma},\g^0(1)) \\
    &= \sum_{v\mid p}[F_v:\Q_p]\dim\b^0 - \sum_{v\mid\infty}h^0(\Gamma_{F_v},\g^0) \\
    &= [F:\Q]\dim\b^0 - \frac{[F:\Q]}2\dim\g^0 \\
    &= \frac{[F:\Q]}2\dim\t^0.
    \end{align*}
    Notice that we would be done if $H^1_{\P^\perp}(\Gamma_{F,\Sigma},\g^0(1)) = 0$.
    Supposing it isn't, choose nonzero $\phi \in H^1_{\P^\perp}(\Gamma_{F,\Sigma},\g^0(1))$ and nonzero $\psi \in H^1_\P(\Gamma_{F,\Sigma},\g^0)$ (which exists by the dimension calculation above).
    By \Cref{ceb-lemma} below (which is a generalization of Lemma 5.3 in \cite{pat}) there exists a prime $w\not\in\Sigma$ such that $\ol\rho|_{\Gamma_{F_w}}$ is of Ramakrishna type, $\phi|_{\Gamma_{F_w}} \not\in L_w^{\Ram,\perp}$, and $\psi|_{\Gamma_{F_w}} \not\in L_w := L^{\unr}_w \cap L_w^{\Ram}$.
    Then $L_w^\perp = L_w^{\unr,\perp}+L_w^{\Ram,\perp}$, and there are inclusions
    \begin{align} 
        H^1_{\P^\perp\cup L_w^{\Ram,\perp}}(\Gamma_{F,\Sigma\cup w},\g^0(1)) 
        &\to H^1_{\P^\perp\cup L_w^{\perp}}(\Gamma_{F,\Sigma\cup w},\g^0(1)) \\
        H^1_{\P^\perp}(\Gamma_{F,\Sigma},\g^0(1)) 
        &\to H^1_{\P^\perp\cup L_w^{\perp}}(\Gamma_{F,\Sigma\cup w},\g^0(1)),
    \end{align}
    the second of which is an isomorphism, only using the assumption that $\psi|_{\Gamma_{F_w}} \not\in L_w$ (see the proof of Prop.~5.2 in \cite{pat}).
    Then we have an exact sequence (all cohomology is applied to $\g^0(1)$):
    \begin{equation}\label{ram-seq}
        0 \to H^1_{\P^\perp\cup L_w^{\Ram,\perp}}(\Gamma_{F,\Sigma\cup w}) \to H^1_{\P^\perp}(\Gamma_{F,\Sigma}) \to H^1(\Gamma_{F_w})/L_w^{\Ram,\perp}
    \end{equation}
    By the assumption $\phi|_{\Gamma_{F_w}} \not\in L_w^{\Ram,\perp}$, 
    \[ h^1_{\P^\perp\cup L_w^{\Ram,\perp}}(\Gamma_{F,\Sigma\cup w},\g^0(1)) < h^1_{\P^\perp}(\Gamma_{F,\Sigma},\g^0(1)).  \]
    % From this point, using Lemma 5.3 of \cite{pat}, the proof of Prop. 5.2 in \cite{pat} shows how the addition of a Ramakrishna condition at a prime $v \not\in \Sigma$ to the deformation theory allows us to reduce the dimension of $H^1_{\P^\perp}(\Gamma_{F,\Sigma},\g^0(1))$, without interfering with the calculation above.
    After a finite number of applications of \Cref{ceb-lemma}, the dual Selmer group vanishes.
    % Let $Q$ denote the set of auxiliary primes introduced by applications of \Cref{ceb-lemma}.
\end{proof}

\begin{rmk}
    Assume that we have added sufficiently many primes to $\Sigma$ such that \[ H^1_{\P^\perp}(\Gamma_{F,\Sigma},\g^0(1)) = 0. \]
    For any nonzero $\psi \in H^1_\P(\Gamma_{F,\Sigma},\g^0)$ we could apply \Cref{ceb-lemma} with $\phi = 0$, choose a prime $w$ satisfying the conclusion, then add a local condition at $w$ to $\P$ as in the proof of \Cref{ax-ram}.
    The proof goes through all the way through \cref{ram-seq}.
    Then we conclude that \[ H^1_{\P^\perp\cup L_w^{\Ram,\perp}}(\Gamma_{F,\Sigma\cup w},\g^0(1)) = 0, \] and $\psi \not\in H^1_{\P\cup L_w}(\Gamma_{F,\Sigma\cup w},\g^0)$.
    In \Cref{sparse}, we will add additional primes of ramification in this way to exclude certain cocycles from the tangent space of certain deformation problems.
\end{rmk}

\begin{lem}\label{ceb-lemma}
    Assume that $\psi \in H^1(\Gamma_{F,\Sigma},\g^0)$ is nonzero and $\phi \in H^1(\Gamma_{F,\Sigma},\g^0(1))$.
\begin{enumerate}
    \item If $\phi\neq 0$, then assume $\psi$ and $\phi$ satisfy axioms (5) and (6) listed above.
    Then there are infinitely many primes $w \not\in \Sigma$ such that $\ol\rho|_{\Gamma_{F_w}}$ is of Ramakrishna type and satisfying
    \[ \phi|_{\Gamma_{F_w}} \not\in L_w^{\Ram,\perp}\quad\text{and}\quad \psi|_{\Gamma_{F_w}} \not\in L^{\unr}_w \cap L_w^{\Ram}. \]
    \item If $\phi = 0$, then assume that $\psi$ satisfies axiom (6) (without any conditions on $\phi$).
    Then there exists infinitely many primes $w\not\in \Sigma$ such that $\ol\rho|_{\Gamma_{F_w}}$ is of Ramakrishna type and $\psi|_{\Gamma_{F_w}} \not\in L^{\unr}_w \cap L_w^{\Ram}.$
\end{enumerate}
\end{lem}

\begin{proof}
    The proof of (1) is identical to the proof of Lemma 5.3 in \cite{pat}, even though we are making slightly more general assumptions.
    The proof of (2) is similar but easier; just ignore all mention of $\phi$.
\end{proof}

\subsection{Representations with large image}
As in \S6 of \cite{pat}, we can satisfy the hypotheses of \Cref{ax-ram} by taking a generous image assumption (up to the existence of some local conditions).
Precisely, we have the following theorem.
Compare to Theorem 6.4 in \cite{pat}.

\begin{prop}\label{large-image}
    Let $F$ be imaginary CM with $[F(\mu_p):F] = p-1$, and let $\ol\rho:\Gamma_{F,\Sigma} \to G(k)$ satisfy
    \begin{enumerate}
        \item There is a subfield $k' \subset k$ such that 
        \[ \ol{(G^{\der})}^{\simpc}(k') \subset \ol\rho(\Gamma_F) \subset Z_G(k)\cdot G(k'). \]
        \item $p-1$ is greater than the maximum of $8\# Z_{(G^{\der})^{\simpc}}$ and
        \[
        \begin{cases}
            (h-1)\#Z_{(G^{\der})^{\simpc}} & \text{if $\#Z_{(G^{\der})^{\simpc}}$ is even; or} \\
            (2h-2)\#Z_{(G^{\der})^{\simpc}} & \text{if $\#Z_{(G^{\der})^{\simpc}}$ is odd.}
        \end{cases}
        \]
        \item For all places $v \in \Sigma$ not dividing $p\cdot\infty$, $\ol\rho_v$ satisfies a liftable local deformation condition $\P_v$ with tangent space of dimension $h^0(\Gamma_{F_v},\g^0)$.
        \item For all places $v\mid p$, $\ol\rho_v$ is nearly ordinary, satisfying the conditions (REG) and (REG*).
    \end{enumerate}
    Then there exists a finite set of primes $Q$ disjoint from $\Sigma$, such that the universal deformation ring $R_{\ol\rho}^{\{\P_v\}_{v\in\Sigma\cup Q}}$ with $\P_v$ as above for $v\in \Sigma$ and $\P_v$ the Ramakrishna condition for $v \in Q$ is isomorphic to $\O\lb X_1,\ldots,X_{\frac{[F:\Q]}2\dim t}\rb$.
\end{prop}

\begin{proof}
    The hypotheses clearly imply axioms (1)-(3) of \Cref{ax-ram}.
    The proofs that axioms (4)-(5) of \Cref{ax-ram} are satisfied is exactly the same as in \S6 of \cite{pat}.
    The proof that axiom (6) holds is the same as the proof of Lemma 6.7 in \cite{pat}, except that we use $x = (4\rho^\vee-\alpha)(t)$ instead of $x = 2\rho^\vee(t)$.
    This ensures that there is a unique root (namely $\alpha$) for which $\alpha\circ\ol\rho(\sigma) = \ol\kappa(\sigma)$.
    If one uses $2\rho^\vee(t)$, then every simple root $\alpha$ satisfies this identity.
\end{proof}

\begin{rmk}
    In fact, one can check that under the large image hypothesis, properties (5) and (6) hold for any cocycles, not only for Selmer classes.
    We will use this in \Cref{sparse}.
\end{rmk}
\section{Weights in the case $G = \GL_2$}\label{gl2-weights}

\subsection{Introduction}
\Cref{large-image} produces universal deformation rings of the desired size and shape, but so far we haven't said anything about the Hodge-Tate weights of its $\ol \Q_p$-points.
We start with the example $G = \GL_2$ and $F/\Q$ quadratic imaginary for concreteness, and to address some subtleties that aren't addressed in \cite{cm}, before considering more general groups $G$ and number fields $F$.

\subsection{Hodge-Tate cocharacter}
Suppose that $\rho:\Gamma_{\Q_p} \to \GL_2(\ol\Q_p)$ is a Hodge-Tate representation of $\Gamma_{\Q_p}$ acting on a 2-dimensional $\ol\Q_p$-vector space $V$, with Hodge-Tate weights $a,b$.
One way to record the weights which generalizes well to more other algebraic groups is to use a Hodeg-Tate cocharacter, unique up to swapping $a,b$ (i.e.~up to Weyl conjugation):
\[ \mu_{HT}:\G_m \to \GL_2,\qquad \mu_{HT}(t) = \bpm t^a & 0 \\ 0 & t^b \epm. \]
We will briefly review Hodge-Tate cocharacters, as they will play a major role in what follows.

\begin{defn}
For a reductive group $G$, and a representation $\rho:\Gamma_{\Q_p} \to G(\ol\Q_p)$ we say that $\rho$ is \textbf{Hodge-Tate} if for every representation $G \to \GL(V)$ the composition $\Gamma_{\Q_p} \to G(\ol\Q_p) \to \GL(V)$ is Hodge-Tate.
\end{defn}

\begin{lem}
If $\rho:\Gamma_{\Q_p} \to G(\ol\Q_p)$ is Hodge-Tate, then there exists a cocharacter $\mu_{HT} \in X_\bullet(T)$ called the \textbf{Hodge-Tate cocharacter} such that for any representation $r:G \to \GL(V)$ of $G$, 
\[ V_{\C_p} = \bigoplus_{i\in\Z}V^{\mu_{HT}(t) = t^i} \]
is the Hodge-Tate decomposition of $r\circ\rho$.
\end{lem}

\begin{proof}
The proof is by Tannakian formalism \cite[II]{dm}.
First, note that $\Rep_{\ol\Q_p}(G) \cong \Rep_{\C_p}(G)$, and $\Rep_{\C_p}(\G_m)$ is isomorphic to the category of $\Z$-graded $\C_p$-vecor spaces.
``Hodge-Tate decomposition'' induces a $\otimes$-functor, which commutes with fiber functors:
\[ 
\begin{tikzcd}[column sep=0]
\Rep_{\C_p}(G) \ar[rr] \ar[dr] && \Rep_{\C_p}(\G_m) \ar[dl] \\
& \Vc_{\C_p}
\end{tikzcd}
\]
Then Tannakian formalism produces a homomorphism
\[ \mu_{HT}:\G_m \to G \]
which satisfies the conclusion of the lemma.
\end{proof}

\begin{rmk}
    $\mu_{HT}$ is only well defined up to conjugation.  If we fix a split maximal torus $T$ of $G$, then $\mu_{HT}$ is a well defined element of $X_\bullet(T)/W_G$, where $W_G$ denote the Weyl group of $(G,T)$.
\end{rmk}

\begin{rmk}
    Similarly we say $\rho$ is \textbf{de Rham} if for every representation $G \to \GL(V)$ the composition $\Gamma_{\Q_p}\xrightarrow{\rho}G(\ol\Q_p)\to GL(V)$ is de Rham.
    Notice that if $\rho$ is de Rham, then it is automatically Hodge-Tate.
\end{rmk}

\subsection{Weights in weight space}
Returning to the case $G = \GL_2$, consider the nearly ordinary setup.
Namely that (perhaps after conjugation by an element of $\wh G$) $\rho: \Gamma_{\Q_p} \to B(\ol\Q_p)$.
If $\rho$ is Hodge-Tate then the result will have the form
\[ \bpm \kappa^a\chi & * \\ 0 & \kappa^b \epm \]
for some integers $a,b$ (the Hodge-Tate weights) and some finite order character $\chi$.
Since the deformations we consider have fixed determinant, one of $a,b$ determines the the other.
Or, both can be determined from the weight of the quotient $\kappa^a\chi/\kappa^b$.

Suppose we are in the setup of \cite[\S7]{cm}, namely we have a quadratic imaginary number field $F/\Q$ with $p$ split, and a representation $\rho: \Gamma_F \to \GL_2(\ol\Q_p)$ which is nearly ordinary at both places $v,\ol v$ dividing $p$.
Let 
\[ \O_{F,p}^\times = (\O_F\otimes \Z_p)^\times \cong \prod_{v\mid p}\O_{F_v}^\times \]
(in the present case $\O_{F,p}^\times\cong \Z_p^\times \times \Z_p^\times$).
Consider the functor from the category of rigid analytic spaces to abelian groups which sends
\[ X \mapsto \Hom(\O_{F,p}^\times,\O_X(X)^\times). \]

\begin{prop}
The functor $X \mapsto \Hom(\O_{F,p}^\times,\O_X(X)^\times)$ is representable. We call the representing rigid space \textbf{weight space}, and denote it as $\W_F$ or just $\W$ when $F$ is clear.
\end{prop}

\begin{proof}
See \cite[\S2]{buzzard}
\end{proof}

Now assume that $\rho_v$ and $\rho_{\ol v}$ have the same Hodge-Tate weights.
Then if
\[ \rho_v = \bpm \chi_{v,1} & * \\ 0 & \chi_{v,2} \epm \qquad \rho_{\ol v} = \bpm \chi_{\ol v,1} & * \\ 0 & \chi_{\ol v,2} \epm, \]
and \[ \chi_v = \chi_{v,1}/\chi_{v,2}, \qquad \chi_{\ol v} = \chi_{\ol v,1}/\chi_{\ol v,2}, \]
then $\chi_v$ either has the same Hodge-Tate weight as $\chi_{\ol v}$ or $\chi_{\ol v}^{-1}$.
Assume for simplicity that $\chi_v$ and $\chi_{\ol v}$ have the same Hodge-Tate weight (the other case is similar).
Then $\chi_v$ and $\chi_{\ol v}$ assemble to produce a $\ol\Q_p$-point in $\W$ in the following way:
since each character $\phi$ of $\Gamma_{\Q_p}$ factors through the abelianization $\Gamma_{\Q_p}^{\ab}$, the restriction of $\phi$ to inertia gives a map out of $I_{\Q_p}^{\ab} \cong \Z_p^\times$, via class field theory.
For example, $\kappa^n$ gives the map $x \mapsto x^n$ (under a suitable normalization of class field theory).
If $a$ is the Hodge-Tate weight of both $\chi_v$ and $\chi_{\ol v}$ then they induce a map
\[ \O_{F,p}^\times \to \ol\Q_p^\times, \qquad x \mapsto x^a\ol x^a\varepsilon(x) = N(x)^a\varepsilon(x), \]
where $x\mapsto x$ and $x\mapsto \ol x$ denote the embeddings $F\to \ol\Q_p$ extended to $\O_{F,p}^\times$, $N:\O_{F,p}^\times \to \Z_p^\times$ is the induced norm map, and $\varepsilon$ is a finite order character.
\begin{defn}
We say a weight $\chi:\O_{F,p}^\times \to \ol\Q_p^\times$ is
\begin{itemize}
    \item \textbf{algebraic} if it is of the form $\chi(x) = x^a\ol x^b$ for integers $a$ and $b$,
    \item \textbf{locally algebraic} if it is the product of an algebraic weight and a finite order character,
    \item \textbf{parallel} if it is algebraic and $a=b$, and
    \item \textbf{locally parallel} if it is the product of a parallel weight and a finite order character.
\end{itemize}
\end{defn}

Let $\W^{\pr}$ denote the space of locally parallel weights, and $\W_0$ the space of parallel weights.

\subsection{A 1-parameter family in weight space}
The deformation theory in \cite[\S7]{cm} produces a universal deformation 
\[ \rho: \Gamma_F \to \GL_2(\O\lb X\rb), \] 
which is nearly ordinary at $v,\ol v$.
Here $\O = W(k)$ is the ring of Witt vectors for the finite field $k$ appearing in the definition of the residual representation.
After suitable conjugation by $\wh G$, the restrictions $\rho_v,\rho_{\ol v}$ for $v,\ol v\mid p$ have image contained in the Borel $B(\O\lb X\rb)$.
Then with the definitions above, $\chi_v$ and $\chi_{\ol v}$ assemble to 
\[ \chi_v\times \chi_{\ol v}: \O_{F,p}^\times \to \O\lb X\rb^\times \hookrightarrow \O_{\D^1}(\D^1)^\times, \]
where $\D^n$ is the rigid open unit $n$-ball.
By definition of $\W$, this produces a map of rigid spaces
\[ \gamma:\D^1 \to \W. \]
We get a similar map $\gamma'$ by using $\chi_v$ and $\chi_{\ol v}^{-1}$.
The claim of \cite[\S7]{cm} is that $\gamma^{-1}(\W^{\pr})$ and $\gamma'^{-1}(\W^{\pr})$ are finite, i.e.~$\rho$ has finitely many parallel weight specializations.
There are a couple subtleties which aren't addressed in \cite{cm}, one of which is explained in \cite{loeffler}, and which we will also explain.
See \cite{loeffler} for more details.

\subsection{The topology of weight space}

In \cite[\S3]{loeffler} it is shown that $\W$ can be described as the generic fiber of the formal scheme $\Spf(\Z_p\lb\O_{F,p}^\times\rb)$, and using \cite[7.1.9]{de-jong} that there is a natural bijection between the $\ol\Q_p$-points of $W := \Spec(\Z_p\lb\O_{F,p}^\times\rb)$ and $\W$.
This allows us to think about the set of points of $\W$ under either a rigid topology or a Zariski topology.
$\W$ is a 2-dimensional space, in fact it is isomorphic to a finite union of open 2-balls (see \cite[Lemma 2]{buzzard}).
Recall that $\W_0$ consists of weights of the form
\[ x \mapsto (x\ol x)^n, \]
in other words weights which factor through the norm map
\[ N:\O_{F,p}^\times \to \Z_p^\times. \]

\begin{lem}\label{gl2-closed-lemma}
    Both the Zariski and rigid closures of $W_0$ in $W(\ol\Q_p)$ are 1-dimensional.
\end{lem}

\begin{proof}
    Let $U$ denote the image of the norm map $N:\O_{F,p}^\times \to \Z_p^\times$.
    Then $U$ is finite index in $\Z_p^\times$, thus 1-dimensional (e.g.~as a $p$-adic Lie group).
    Thus, the norm induces a surjective homomorphism
    \[ \Z_p\lb\O_{F,p}^\times\rb \twoheadrightarrow \Z_p\lb U\rb, \]
    and the inclusion of a 1-dimensional (Zariski) closed subspace on generic fibers
    \[ \Spec(\Z_p\lb U\rb)(\ol\Q_p) \hookrightarrow W(\ol\Q_p). \]
    Arbitrary characters of $U$ can be approximated by characters of the form $x\mapsto x^n$ with $n\in\Z$, so $\W_0$ is dense for both topolgies in the image of $\Spec(\Z_p\lb U\rb)(\ol \Q_p)$ in $W(\ol\Q_p)$.
\end{proof}

\begin{prop}\label{gl2-closed}
In the Zariski topology, the closure of $\W^{\pr}$ is a 2-dimensional subspace of $W(\ol\Q_p)$.
In the rigid analytic topology, the closure of $\W^{\pr}$ is a 1-dimensional subspace of $\W$.
\end{prop}

\begin{proof}
\textit{(See also \cite[Prop.~4.2]{loeffler})}
For any finite character $\varepsilon:\O_{F,p}^\times \to \ol\Q_p^\times$ (of which there are infinitely many), define $\W_\varepsilon$ as the subspace of $\W^{\pr}$ consiting of weights of the form
\[ x\mapsto (x\ol x)^n\varepsilon(x). \]
Then $\W^{\pr}$ is the disjoint union of the $\W_\varepsilon$ for all $\varepsilon$, where $\W_0$ corresponds to the trivial character.
$\W_\varepsilon$ is a disjoint translate of $\W_0$, thus has closure isomorphic to (but disjoint from) the closure of $\W_0$.
Thus the Zariski closure of $\W^{\pr}$ contains an infinite collection of disjoint 1-dimensional subspaces, and therefore cannot be 1-dimensional.

Consider the element
\[ u = (1+p,(1+p)^{-1}) \in \Z_p^\times\times\Z_p^\times \cong \O_{F,p}^\times. \]
Then $N(u) = 1$, and $u$ is a $\Z_p$-generator for the norm 1 elements of $\O_{F,p}^\times/$torsion.
Then the map $\chi \mapsto \log_p(\chi(u))$ is a rigid analytic function on $\W$, which cuts out a 1-dimensional closed subscheme.
But any weight $\chi \in \W^{\pr}$ is in this subsceme, since any such weight can be written as $\chi(x) = N(x)^n\varepsilon(x)$, and satisfies
\[ \log_p(\chi(u)) = n\log_p(N(u)) + \log(\varepsilon(u)) = n\log_p(1) + 0 = 0, \]
since $\varepsilon$ takes values in the roots of unity, which are 0's of the $p$-adic logarithm.
\end{proof}

\begin{rmk}
    The claim in \cite{cm} is that $\gamma^{-1}(\W^{\pr})$ is 0-dimensional \textit{in the Zariski topology} (i.e.~\textit{finite}).
    Their strategy only allows one to deduce that $\gamma^{-1}(\W^{\pr})$ 0-dimensional in the \textit{rigid topology} (which does not imply finite).
    But this is still good enough to imply that $\gamma^{-1}(\W^{\pr})$ is at least sparse and discrete.
\end{rmk}

\subsection{Enforcing finiteness}
If one desires finiteness, one possible solution is to put some sort of restriction on the allowed finite order characters.
The most natural way to do this is to constrain the ramification, which could be done in a number of ways.
For example, one suggestion of \cite{loeffler} is to only consider specializations of $\rho$ which are crystalline at both $v,\ol v$, or become so after any fixed finite extension of $\Q_p$.
Another solution would be to only consider $K$-points, where $K/\Q_p$ is an extension which doesn't contain all $p$-power roots of unity, so as to limit the possible ramification of characters.

\subsection{Passage to infinitesimal weights}
Next, I will explain the use of ``infinitesimal weights'' in \cite{cm} to show that $\gamma^{-1}(\W^{\pr})$ is 0-dimensional .
% (and symmetrically that $\gamma'^{-1}(\W^{\pr})$ 0-dimensional).

\begin{defn}
For a cocycle $\varphi \in H^1(\Gamma_{F_v},\b^0)$, the \textbf{infinitesimal weight} of $\varphi$ is given by applying to $\varphi$ the composition of maps
\[
\begin{tikzcd}
H^1(\Gamma_{\Q_p},\b^0) \ar[r] \ar[dr,"\beta",dashed] & H^1(\Gamma_{\Q_p},\b^0/\n) \ar[r,"\cong"] & \Hom(\Gamma_{\Q_p}^{\ab},\t^0) \ar[d,"\text{res + CFT}"] \\
& k & \Hom(\Z_p^\times,k) \ar[l,"ev_{1+p}"']
\end{tikzcd}
\]
where the first map is induced by $\b^0 \to \b^0/\n$, the second follows from the fact that $\b^0/\n \cong \t^0$ is abelian with trivial action, the third restricts to inertia and uses class field theory, and the fourth evaluates at $1+p$.
\end{defn}

\begin{prop}\label{gl2-prop}
    One can associate infinitesimal weights to $\rho$ at $v$ and $\ol v$, such that either 
    \begin{enumerate}
        \item the infinitesimal weights of $\rho$ are parallel (equal at $v$ and $\ol v$), or
        \item $\gamma^{-1}(\W^{\pr})$ is 0-dimensional (in the rigid topology), and $\gamma^{-1}(\W_\varepsilon)$ is finite for each $\varepsilon$.
    \end{enumerate} 
\end{prop}

\begin{proof}
    Let $\chi_v,\chi_{\ol v}:\Z_p^\times \to \O\lb X\rb^\times$ be defined as above, by taking the quotient of the diagonal characters of $\rho_v$ and $\rho_{\ol v}$ respectively.
    Let $f_v(X) = \chi_v(1+p)$ and $f_{\ol v}(X) = \chi_{\ol v}(1+p)$.
    The weight of $\chi_v$ at a specialization $\alpha \in \ol \Q_p$ with $|\alpha|_p<1$ of $X$ can be extracted as
    \[ \frac{\log_p(f_v(\alpha))}{\log_p(1+p)}, \]
    and similarly for $\ol v$.
    If the weights of $\chi_v$ and $\chi_{\ol v}$ are equal at $\alpha$, then we can conclude that
    \[ \log_p\left(\frac{f_v(\alpha)}{f_{\ol v}(\alpha)}\right) = 0, \]
    Since $f_v(\alpha),f_{\ol v}(\alpha)\in \O^\times$, $f_v(\alpha)/f_{\ol v}(\alpha)$ must be a root of unity.
    in other words, if 
    \[ g(X) := f_v(X)/f_{\ol v}(X) \in \O\lb X\rb^\times, \] 
    then $g(\alpha)$ is a root of unity.
    If $\chi_v$ and $\chi_{\ol v}$ have equal weights (meaning that the weights are equal at all specializations $\alpha)$, then $g$ only takes values in the roots of unity.
    I claim that this implies that $g$ is constant.
    
\begin{lem}
    Let $\O$ be a complete discrete valuation ring with uniformizer $\varpi$.
    If $g \in \O\lb X\rb$ and $g(\alpha)$ is a root of unity for any $\alpha \in \ol {\O[\varpi^{-1}]}$ with $|\alpha|<1$, then $g$ is constant.
\end{lem}
\begin{proof}[Proof of Lemma]
    Fix any finite extension $M/\O[\varpi^{-1}]$.
    Then $M$ contains a finite number of roots of unity.
    For $\zeta \in \mu(M)$, $g(X) - \zeta$ has finitely many zeros in the maximal ideal of $\O_M$ (by compactness and the isolation of zeros of power series), unless $g(X) = \zeta$ on all specializations in $M$.
    Since every element of the maximal ideal of $\O_M$ must be such a zero and $(\varpi)$ is an infinite set, we must have $g(X) = \zeta$ on $M$-points for some $\zeta \in \mu(M)$.
    Limiting over all finite extensions $M$, we see that $g$ must be constant.
\end{proof}

Continuing the proof of the theorem, we see that if $\chi_v$ and $\chi_{\ol v}$ have equal weights, then $f_v(X) = \zeta f_{\ol v}(X)$ for some root of unity $\zeta$.
Otherwise, $g(X) = \zeta$ can only have finitely many zeros for \textit{every} root of unity $\zeta$, which implies that $\gamma^{-1}(\W_\varepsilon)$ is finite for each $\varepsilon$, i.e.~(2) holds.

Suppose that (2) does not hold.
Then we may assume that $f_v(X) = \zeta f_{\ol v}(X)$ for some root of unity $\zeta$.
Let $\varpi \in \O$ denote a uniformizer, and write
\[ f_v(x) \equiv \sum_{n=0}^\infty a_n^vX^n \pmod \varpi, \]
and similarly for $\ol v$.
Note that $\O\lb X\rb/(\varpi,X^2) \cong k[\epsilon]$ by $X \mapsto \epsilon$.
We have
\[ f_v(X) \equiv a_0^v+a_1^v\epsilon \mod{(\varpi, X^2)} \]
for some $a_0^v,a_1^v \in k$, and similarly for $f_{\ol v}(X)$.

Notice that the composition
\[ \rho_\epsilon: \Gamma_F \xrightarrow{\ \rho\ } \GL_2(\O\lb X\rb) \to \GL_2(k[\epsilon]) \]
defines a class in $\Def_{\ol\rho}^{\P}(k[\epsilon]) \cong H^1_{\P}(\Gamma_{F,\Sigma},\g^0)$.
Restricting to $\Gamma_{F_v}$ or $\Gamma_{F_{\ol v}}$ gives cocycles
\[ \varphi_v \in H^1(\Gamma_{F_v},\b^0),\qquad \varphi_{\ol v} \in H^1(\Gamma_{F_{\ol v}},\b^0). \]
Tracing through the definitions of the isomorphism $\Def_{\ol\rho}^{\P}(k[\epsilon]) \cong H^1_{\P}(\Gamma_{F,\Sigma},\g^0)$ and the infinitesimal weight map $\beta$ shows that $\beta(\varphi_v) = a_1^v/a_0^v$ and $\beta(\varphi_{\ol v})=a_1^{\ol v}/a_0^{\ol v}$.
Notice that since $f_v(X)$ is a constant multiple of $f_{\ol v}(X)$, $a_1^v/a_0^v = a_1^{\ol v}/a_0^{\ol v}$, thus (1) holds.
\end{proof}

\begin{rmk}
In order to consider all specializations of $\rho$ with parallel weight, we should re-prove \Cref{gl2-prop} using $\chi_v$, $\chi_{\ol v}^{-1}$, and $\gamma'$ as well (the proof is the same).
This is the case where $\rho_v$ and $\rho_{\ol v}$ have the same weights, but in opposite diagonal matrix entries.
Then the infinitesimal weights will be negatives of each other.
Combining both results into one:
\end{rmk}

\begin{cor}
    Associating inifinitesimal weights to $\rho$ as in \Cref{gl2-prop}, either
    \begin{enumerate}
        \item the infinitesimal weights of $\rho$ are either equal or negatives of each other, or
        \item $\gamma^{-1}(W^{\pr})\cup \gamma'^{-1}(W^{\pr})$ is 0-dimensional (in the rigid topology), and $\gamma^{-1}(W_\varepsilon)\cup \gamma'^{-1}(W_\varepsilon)$ is finite for each $\varepsilon$.
    \end{enumerate}
\end{cor}

\begin{rmk}\label{gl2-concern}
    To arrive at the conclusion of Theorem 7.1 in \cite{cm}, we wish end up in case (2) \textit{of the corollary} rather than the proposition, and this leads to another subtlety that doesn't seem to be addressed in \cite{cm}.
    To conclude the proof of Theorem 7.1 of \cite{cm}, the idea is that if you are in case (1) of \Cref{gl2-prop}, then add an additional prime of ramification to ensure you are no longer in case (1).
    The argument allows one to avoid any subspace in the infinitesimal weight space $k\oplus k$, but in light of the corollary, one really wishes to avoid both the diagonal subspace and the anti-diagonal subspace, which together span all of infinitesimal weight space.
    Unfortunately, the argument in \cite{cm} doesn't rule out the possibility that every attempt to move out of the diagonal subspace places you into the anti-diagonal subspace, or vice versa.
\end{rmk}

\subsection{Geometric lifts} We really don't care about specializations of $\rho$ with parallel weight as much as we care about specializations which are automorphic, which must also be \textit{geometric}, in the sense of the Fontaine-Mazure conjecture.
We can use this extra structure to address the concerns raised in \Cref{gl2-concern}, by making one additional minor assumption on $\ol\rho$, namely that the $\ol\rho_v$ and $\ol\rho_{\ol v}$ are non-split.
For then $\rho_v$ and $\rho_{\ol v}$ will also be non-split, and all but finitely many specializations of $\rho_v$ and $\rho_{\ol v}$ will be non-split.
If a non-split specialization of $\rho_v$ is de Rham (e.g.~if $\rho$ is geometric), then it induces a non-zero class in the group $H^1_g(\Gamma_{F_v},\Q_p(\chi_v))$ defined by Block and Kato in \cite{bk}.
The dimensions of $H^1_g$ can be calculated using basic $p$-adic Hodge theory, for example in  \S1 of \cite{nek}.
Proposition 1.24 of \cite{nek} gives the formula for the dimension of $H^1_g$.
From this we conclude

\begin{prop}\label{p-bloch-kato}
    If $H^1_g(\Gamma_{F_v},\Q_p(\chi_v)) \neq 0$, then the Hodge-Tate weight of $\chi_v$ is a non-negative integer.
\end{prop}

Since $\rho_v$ or $\rho_{\ol v}$ will only be split at finitely many specializations, we can ignore these and consider the specializations of $\rho$ where $\rho_v$ and $\rho_{\ol v}$ are both non-split.
The proposition implies that in this case the weights of $\chi_v$ and $\chi_{\ol v}$ (rather than just $\chi_{\ol v}^{\pm 1}$) must agree on any geometric specialization of $\rho$.
% In particular, if $\rho$ has parallel weight, then the weights of $\chi_v$ and $\chi_{\ol v}$ (rather than $\chi_{\ol v}^{-1}$) are equal at all but finitely many (where $\rho_v$ or $\rho_v$ splits) of the geometric specializations of $\rho$.
In particular, we no longer need to worry about the anti-diagonal infinitesimal weight space, and \Cref{gl2-prop} together with the arguments of \cite{cm} suffice to conclude:

\begin{thm}
    (Compare with \cite[Thm.~7.1]{cm})
    Let $k/\F_p$ be a finite field, $F/\Q$ a quadratic imaginary extension in which $p$ splits.
    Let $\ol\rho:\Gamma_F \to \GL_2(k)$ denote a continuous absolutely irreducible Galois representation satisfying:
    \begin{enumerate}
        \item The image of $\ol\rho$ contains $\SL_2(k)$, and $p\geq 5$,
        \item If $v\mid p$, then $\ol\rho$ is nearly ordinary at $v$ taking the shape
        \[ \rho|_{I_v} = \bpm \psi & * \\ 0 & 1\epm \]
        with $\psi \neq 1,\kappa^{\pm 1}$ and $*\neq 0$.
        \item If $v\nmid p$ and $\ol\rho$ is ramified at $v$, then $H^2(\Gamma_{F_v},\g^0) = 0$.
    \end{enumerate}
    Then there exists a Galois representation $\rho:\Gamma_F \to \GL_2(W(k)\lb X\rb)$ lifting $\ol\rho$ such that
    \begin{enumerate}
        \item $\ol\rho$ is unramified outside some finite set of primes $\Sigma$, and if $v\mid p$ then $\rho|_{\Gamma_{F_v}}$ is nearly ordinary, and
        \item the set of specializations of $\rho$ which are geometric with parallel Hodge-Tate weights is 0-dimensional in the rigid topology.
    \end{enumerate}
\end{thm}
\section{Hodge symmetry and parallel weight}\label{symmetry}

\subsection{Introduction} 
This section gives automorphic and motivic analogues of ``parallel Hodge-Tate weights'' in the $\GL_2$ case and extends these ideas to more general groups $G$.
Then we can study weights, weight space, parallel weight space, etc.~in this generality.

\subsection{Vanishing of cuspidal cohomology}
Irreducible Galois representations $\Gamma_F \to \GL_2(\ol\Q_p)$ for $F/\Q$ quadratic imaginary are conjecturally connected to cuspidal automorphic forms on $\GL_2(\A_F)$.
If such an automorphic form is ``cohomological,'' then it appears in a cohomology group of the form
\[ H^1_P(\Gamma,S_{a,b}) \]
where $\Gamma$ is a congruence subgroup of $\SL_2(\O_F)$ and $S_{a,b}$ is defined as
\[ S_{a,b} = \Sym^a(\O_F^2) \otimes \ol{\Sym^b(\O_F^2)}, \]
where the action on the second factor is twisted by complex conjugation.
The $P$ refers to ``parabolic'' (or ``cuspidal'') cohomology.
The motivation for ``parallel weight'' is the following result.

\begin{thm}
    $H_P^1(\Gamma,S_{a,b}) = 0$ unless $a = b$.
\end{thm}

For more on this, see \cite{taylor}, which refers the reader to \cite{harder} for details.

\subsection{Vanishing of cohomology in higher rank}
For more general $G$, the book \cite{bw} proves similar vanishing results, at least in the case of a cocompact arithmetic subgroups $\Gamma$.

\begin{theorem}{\cite[VII 6.7]{bw}}
    For $G$ connected, semi-simple, finite center, and no compact factor, $\Gamma$ a cocompact discrete subgroup, and $F$ and irreducible finite dimensional $G$-module with highest weight $\Lambda - \rho$.
    If $\theta\Lambda \neq \Lambda$, then $H^*(\Gamma,F) = 0$.
\end{theorem}

In this theorem $\rho$ is the half sum of the positive roots, and $\theta$ is a fixed Cartan involution.
This theorem is deduced from the following vanishing result for $(\g,\k)$-cohomology, and the decompositon of the cohomology of $\Gamma$ in terms of unitary representations.

\begin{theorem}{\cite[II 6.12]{bw}}
    Let $F$ be the irreducible finite dimensional $\g_\C$-representation with highest weight $\Lambda-\rho$.
    Let $(\pi,H)$ be a unitary $(\g,\k)$-module with Casimir element operating by a scalar operator.
    If $\theta\Lambda \neq \Lambda$, then $\Ext^*_{\g,\k}(F,H) = 0$.
\end{theorem}

\subsection{$\theta\Lambda = \Lambda$ when $\g$ is complex}
For $\g$ the Lie algebra of a \textit{complex} group $G$ (which is what we are looking at since the fields we work over are CM), we have
\[ \g_\C = \g \otimes_\R\C \cong \g \oplus \g, \]
where the action on second factor is twisted by conjugation.
A highest weight $\Lambda$ for $\g_\C$ can then be thought of as a pair of highest weights $\Lambda = (\lambda_1,\lambda_2)$ for each factor of $\g$.

\begin{prop}
$\theta$ ``swaps and dualizes $\Lambda = (\lambda_1,\lambda_2)$'', i.e.~
\[ \theta(\lambda_1,\lambda_2) = (-w_0\lambda_2,-w_0\lambda_1). \]
\end{prop}

\begin{proof}
Since we are assuming $G_{/\C}$ is a reductive algebraic group, it is linear.
Let $\g$ denote the Lie algebra, and choose an embedding $\g \hookrightarrow \gl_n(\R)$, inducing an embedding $\g_\C \hookrightarrow \gl_n(\C)$.
Then by \cite[Prop.~6.28]{knapp} $\theta$ can be chosen as negative-conjugate-transpose on $\g_\C$:
\[ \theta(X\otimes z) = (-\,^tX)\otimes \ol z. \]
In terms of the isomorphism $\g_\C \cong \g\oplus \g$, this says
\[ \theta(X,Y) = (-\,^tY,-\,^tX). \]
Thus on highest weight representations, $\theta$ swaps and dualizes, and in general the dual of a highest weight $\lambda$ representation is a highest weight $-w_0\lambda$ representation.
\end{proof}

Thus the condition that $\theta\Lambda = \Lambda$ amounts to $\lambda_1 = -w_0\lambda_2$.
In the case of a self dual weight (so any weight if $-1 \in W_G$), this says that $\lambda_1 = \lambda_2$, which could obviously be described as ``parallel weight.''

% This is the natural generalization of parallel weight from the $\GL_2$ case to general $G$.

\subsection{Hodge symmetry}
The motivic analogue of what we have been referring to as ``parallel weight'' is essentially Hodge symmetry.
We will give a quick explanation, and refer the reader to \cite[\S2]{pat-ks} for more details.
Let $\mathcal M_{F,E}$ denote a Tannakian category of pure motives over $F$ with coefficients in $E$.
Fixing embeddings $F \hookrightarrow \ol E$ and $\ol E \hookrightarrow \C$ induces a Betti fiber functor $\mathcal M_{F,E} \to \Vc_E$ making $\mathcal M_{F,E}$ into a neutral Tannakian category.
Let $\mathcal G$ denote the Tannakian group, so that the fiber functor induces an isomorphism $\mathcal M_{F,E} \cong \Rep(\mathcal G)$.
Then for each $\tau: F \hookrightarrow \ol E$, Hodge decomposition gives a functor
\[ \mathcal M_{F,E} \to \Rep(\G_m), \]
which induces by Tannakian formalism a $\tau$-labeled Hodge cocharacter $\mu_\tau:\G_m \to \mathcal G$, well defined up to conjugacy.
The weight grading on $\mathcal M_{F,E}$ also induces a (central) cocharacter $\omega$ on $\mathcal G$, and Hodge symmetry implies that for any choice of complex conjugation $c$, 
\[ [\mu_\tau] = \omega - [\mu_{c\tau}], \]
where the brackets denote Weyl conjugacy classes (additive notation is used for consistency with the next section).
Then the comparison isomorphisms of $p$-adic Hodge theory imply the analogue of this statement for Hodge-Tate cocharacters of Galois representations arising in the \'etale cohomology of motives.
Suppose $F$ is a CM field with maximal totally real subextension $F^+$, $v\mid p$ is a place of $F^+$ which splits $v = w\ol w$ in $F$, and $\rho:\Gamma_F \to G(\ol\Q_p)$ is a Galois representation belonging to a compatible system coming from the \'etale cohomology of a motive in $\mathcal M_{F,E}$.
Then each embedding $\tau:F \hookrightarrow \ol\Q_p$ has a complex conjugate embedding $\ol\tau:F\hookrightarrow \ol\Q_p$ (swapping $w$ and $\ol w$), and we have
\[ [\mu_{HT,\tau}^w] = \omega - [\mu_{HT,\ol\tau}^{\ol w}] \]
for some central weight cocharacter $\omega$.
For the definition of $\tau$-labelled Hodge-Tate cocharacters, see \S\ref{weights}.
\section{Weights for general $G$}\label{weights}

\subsection{Introduction} In this section we imitate the $\GL_2$ example of \Cref{gl2-weights} to investigate weights, weight space, parallel weights, and infinitesimal weights when $G$ is a general reductive group, and $F$ is a CM number field satisfying certain hypotheses.

\subsection{Parallel weight for general $G$, and quadratic imaginary $F$}
Suppose that $G$ is of semisimple rank $d$, and that $F/\Q$ is quadratic imaginary with $p$ split.
Consider a representation
\[ \rho: \Gamma_F \to G(\ol\Q_p) \]
which is nearly ordinary at $v,\ol v \mid p$ for a fixed Borel $B$ containing a split maximal torus $T$ and having unipotent radical $N$.
Let $\mu_{HT}^v$ and $\mu_{HT}^{\ol v}$ denote choices of Hodge-Tate cocharacters for $\rho_v:=\rho|_{\Gamma_{F_v}}$ and $\rho_{\ol v}:=\rho|_{\Gamma_{F_{\ol v}}}$, where $p = v\ol v$.

\begin{defn}
The representation $\rho$ has \textbf{parallel Hodge-Tate weight} if there exists a central cocharacter $\omega$ such that
\[ [\mu_{HT}^v] = \omega - [\mu_{HT}^{\ol v}]. \]
The brackets denote Weyl conjugacy classes, and $\omega$ is called the \textit{weight cocharacter}.
\end{defn}

\begin{rmk}
    In general, $\mu_{HT}^v$ and $\mu_{HT}^{\ol v}$ are only determined up to Weyl conjugacy.
    However, just as in the $\GL_2$ case, if we require $\rho_v$ and $\rho_{\ol v}$ to be \textit{non-split} and \textit{geometric}, then we can use $p$-adic Hodge theory to canonically choose $\mu_{HT}$ and $\mu_{HT}^{\ol v}$.
\end{rmk}

\subsection{Non-split representations}\label{ss-non-split}

Fix a pinning of $G$ with Borel subgroup $B$, and suppose that for $v\mid p$ that $\rho_v$ is nearly ordinary with respect to $B$.
The Lie algebra $\b$ of $B$ admits a $B$-stable filtration by root height:
\[ F^r\b = \bigoplus_{ht(\alpha)\geq r}\g_\alpha. \]
For $r\geq 1$, let $N_r$ denote the closed normal subgroup of $B$ with Lie algebra $F^r\b$ (thus $N := N_1$ is the unipotent radical of $B$).
Fixing a spliting $B = N\rtimes T$, we can write
\[ \Gamma_{F_v} \xrightarrow{\rho_v} B(\ol\Q_p) \twoheadrightarrow (B/N_2)(\ol\Q_p) \cong (N/N_2\rtimes T)(\ol\Q_p) \]
as the semi-direct product of $\exp(\phi)$ with $\Gamma_{F_v} \to B/N \cong T$ for some $\phi:\Gamma_{F_v}\to F^1\b/F^2\b$.
Further, $\phi$ defines a class
\[ [\phi] \in H^1(\Gamma_{F_v},F^1\b/F^2\b) \cong H^1(\Gamma_{F_v},\bigoplus_{\alpha \in \Delta}\g_\alpha) \cong \bigoplus_{\alpha \in \Delta} H^1(\Gamma_{F_v},\ol\Q_p(\alpha\circ\rho_v)). \]
For each $\alpha \in \Delta$, let $[\phi]_\alpha$ denote the projection of $[\phi]$ onto $H^1(\Gamma_{F_v},\ol\Q_p(\alpha\circ \rho_v))$.

\begin{defn}\label{d-non-split}
    % Fix a pinning of $G$.
    % Suppose that $\rho_v$ is nearly ordinary with respect to the Borel $B$ associated to the fixed pinning.
    Say that $\rho_v$ is \textbf{non-split} if $[\phi]_\alpha$ is non-zero for every $\alpha \in \Delta$. 
\end{defn}

By \Cref{p-bloch-kato}, we have the following.

\begin{prop}\label{geom-prop}
    Suppose that $\rho$ is nearly ordinary, non-split, and de Rham at $v\mid p$, taking values in a choice of Borel $B$.
    Then for each simple root $\alpha \in \Delta$, $\alpha\circ \rho_v$ has non-negative Hodge-Tate weight.
    % Thus we may choose $\mu_{HT}$ to lie in the dominant Weyl chamber associated to our choice of $B$, and then that the Hodge-Tate weight of $\alpha\circ \rho_v$ is given by $\langle\mu_{HT}^v,\alpha\rangle \geq 0$.
\end{prop}

% \begin{proof}
%     We can once again translate the non-splitness into Galois cohomology.
%     One way to state this is to consider the root $\SL_2 \to G$ associated to $\alpha$, for which
%     \[ \bpm 1 & * \\ 0 & 1 \epm \mapsto U_\alpha \]
%     isomorphically.
%     Then $*$ gives a non-zero class in $H^1_g(\Gamma_{\Q_p},\ol\Q_p(\alpha\circ \rho_v))$, but this group vanishes if $\alpha\circ\rho_v$ has negative Hodge-Tate weight.
% \end{proof}

Let $\Delta = \{\alpha_1,\ldots,\alpha_d\}$ denote the simple roots corresponding to the choice of $B$ and $T$.
For $i = 1,\ldots, d$, let $\chi_{v,i}:\Gamma_{\Q_p} \to \ol\Q_p^\times$ be the character $\chi_{v,i} = \alpha_i \circ \rho_v$, and define $\chi_{\ol v,i}$ similarly.

\begin{cor}
    Assume that $\rho$ has parallel weight, and satisfies the hypotheses of \Cref{geom-prop} at each of $v$ and $\ol v$ for Borel subgroups $B$ and $B'$ respectively.
    Let $w \in W_G$ such that $wB' = B^{op}$.
    Then $\chi_{v,i}$ and $-w\chi_{\ol v,i}$ have the same Hodge-Tate weight for each $i = 1,\ldots, d$.
\end{cor}

\begin{rmk}
    $W_G$ acts on $\chi_{\ol v,i}$ by $(w\cdot \chi_{v,i})(\sigma) = \alpha_i(w\cdot \pi(\rho_v(\sigma)))$, where $\pi:B \to T$ is the projection map.
\end{rmk}

\begin{rmk}
    % If $\rho_v$ and $\rho_{\ol v}$ are ordinary for different Borel subgroups $B$ and $B'$, then one could replace $w_0$ by an element of the Weyl group $w$ satisfying $wB' = B^{op}$.
    For simplicity, we will always assume $B = B'$, so that $w = w_0$.
\end{rmk}

\begin{proof}[Proof of Corollary]
    The hypotheses and \Cref{geom-prop} imply that if we choose $\mu_{HT}^v$ and $\mu_{HT}^{\ol v}$ such that
    \[ \langle\mu_{HT}^v,\alpha_i\rangle = \HT(\chi_{v,i}) \]
    and similarly for $\ol v$, then $\mu_{HT}^v$ and $\mu_{HT}^{\ol v}$ both lie in the dominant Weyl chamber.
    Since $\rho$ has parallel weight, $[\mu_{HT}^v] = \omega - [\mu_{HT}^{\ol v}]$ for a central character $\omega$.
    Notice that
    \[ \langle\omega-w_0\mu_{HT}^{\ol v},\alpha_i\rangle = \langle-w_0\mu_{HT}^{\ol v},\alpha_i\rangle \geq 0, \]
    by regularity and since $-w_0\mu_{HT}^{\ol v}$ also lies in the dominant Weyl chamber.
    But since $\mu_{HT}^v$ is conjugate to $\omega - w_0\mu_{HT}^{\ol v}$, and both lie in the dominant Weyl chamber,
    % Since $-w_0\mu_{HT}^{\ol v}$ is also in the same Weyl chamber as $\mu_{HT}^{\ol v}$, we have
    \[ \mu_{HT}^v = \omega - w_0\mu_{HT}^{\ol v}. \]
    % Since $\omega$ is central, $\langle\omega,\alpha\rangle = 0$ for each root $\alpha$.
    % In particular, $\chi_{v,i}$ and $-w_0\chi_{\ol v,i}$ have the same Hodge-Tate weight for each $i = 1,\ldots, d$.
    Thus,
    \[ \HT(\chi_{v,i}) = \langle\mu_{HT}^v,\alpha_i\rangle = \langle\omega - w_0\mu_{HT}^{\ol v},\alpha_i\rangle = \HT(-w_0\chi_{\ol v,i}). \]
\end{proof}

\begin{rmk}
    If $-1 \in W_G$, then this says $\chi_{v,i}$ and  $\chi_{\ol v,i}$ have the same Hodge-Tate weight for each $i = 1,\ldots, d$, an exact generalization of parallel weight in the $\GL_2$ case.
\end{rmk}

\subsection{Passage to infinitesimal weight for general groups}

Continue to suppose that $F/\Q$ is quadratic imaginary with $p$ split.
Let $G$ be a split reductive group.
Let $d = \rank G^{\der}$.
Suppose that
\[ \rho: \Gamma_F \to G(\O\lb X_1,\ldots,X_d\rb) \]
is a universal deformation, as constructed in \Cref{large-image}.
Consider
\[ \rho_v,\rho_{\ol v}:\Gamma_{\Q_p} \to G(\O\lb X_1,\ldots,X_d\rb). \]
By the nearly ordinary assumption, we can assume (after conjugation by some element of $\wh G$) that the images are in $B$, where $B$ is a fixed Borel of $G$ containing a split maximal torus $T$ and having unipotent radical $N$.
Let $\Delta = \{\alpha_1,\ldots,\alpha_d\}$ be the simple roots of $(G,B,T)$.

Let $\chi_{v,i}:\Z_p^\times \to \O\lb X_1,\ldots, X_d\rb^\times$ be defined by $\chi_{v,i} = \alpha_i \circ \rho_v$, and define $\chi_{\ol v,i}$ similarly.

For each $i = 1,\ldots,d$
\[ \chi_{v,i}\times (-w_0\chi_{\ol v,i}) : \O_{F,p}^\times \to \O\lb X_1,\ldots, X_d\rb^\times \hookrightarrow \O_{\D^d}(\D^d)^\times \] gives a map 
\[ \gamma_i:\D^d \to \W. \]

\begin{defn}
For a cocycle $\varphi \in H^1(\Gamma_{F_v},\b^0)$ the \textbf{infinitesimal weight} of $\varphi$ is given by applying to $\varphi$ the composition of maps
\[
\begin{tikzcd}
H^1(\Gamma_{\Q_p},\b^0) \ar[r] \ar[dr,"\beta",dashed] & H^1(\Gamma_{\Q_p},\b^0/\n) \ar[r,"\cong"] & \Hom(\Gamma_{\Q_p}^{\ab},\t^0) \ar[d,"\text{res + CFT}"] \\
& k^d & \Hom(\Z_p^\times,k^d) \ar[l,"ev_{1+p}"']
\end{tikzcd}
\]
Say that $\varphi_v \in H^1(\Gamma_{F_v},\b^0)$ and $\varphi_{\ol v} \in H^1(\Gamma_{F_{\ol v}},\b^0)$ have \textbf{parallel infiniteseimal weight} if $\beta(\varphi_v) = -w_0\beta(\varphi_{\ol v})$, where the action of the Weyl group on $k^d$ is induced by the action on $\t$.
\end{defn}

\begin{prop}
    One can associate a collection of infinitesimal weights to $\rho$ at $v$ and $\ol v$ such that either
    \begin{enumerate}
        \item the infinitesimal weights of $\rho$ are parallel, or 
        \item There is an $1\leq i \leq d$ such that $\gamma_i^{-1}(\W_\varepsilon)$ has codimension $\geq 1$ for each $\varepsilon$, and $\gamma_i^{-1}(W^{\pr})$ has codimension $\geq 1$ in the rigid topology.
    \end{enumerate}
\end{prop}

\begin{proof}
Let
\[ f_{v,i} := \chi_{v,i}(1+p),\qquad f_{\ol v,i} := -w_0\chi_{\ol v,i}(1+p). \]
As in the $\GL_2$ case, if $\rho_v$ and $\rho_{\ol v}$ have the same weight for some specialization $\alpha$ (of $X_1,\ldots,X_d$), then
\[ \frac{f_{v,i}(\alpha)}{f_{\ol v,i}(\alpha)} = \zeta \]
for some root of unity $\zeta$.
If $f_{v,i}/f_{\ol v,i} \neq \zeta$ as series, then for each specialization of $X_1,\ldots, X_{d-1}$, there are finitely many specializations of $X_d$ which are roots of $f_{v,i}/f_{\ol v,i} - \zeta$, so $\gamma_i^{-1}(\W_\varepsilon)$ is $\leq d-1$ dimensional for each finite character $\varepsilon$.
Thus (2) holds.

If for every $1\leq i\leq d$ we have
\[ f_{v,i} = \zeta_if_{\ol v,i} \]
for some roots of unity $\zeta_i$, then we can form the tuples $(f_{v,1},\ldots,f_{v,d})$ and $(f_{\ol v,1},\ldots,f_{\ol v,d})$ and reduce down to $k[\epsilon]$ in various ways using the quotient maps
\[ \O\lb X_1,\ldots,X_d\rb \twoheadrightarrow \O\lb X_1,\ldots,X_d\rb/(\pi,X_1,\ldots,X_{j-1},X_j^2,X_{j+1},\ldots,X_d) \]
for $j = 1,\ldots,d$.
We have
\[ f_{v,i} \equiv a_{0,i,j}^v + a_{1,i,j}^v\epsilon \mod{(\pi,X_1,\ldots,X_j^2,\ldots,X_d)} \]
for each $j$, and similarly for $\ol v$.
Since $f_{v,i}$ and $f_{\ol v,i}$ differ by a scalar, $a_{1,i,j}^v/a_{0,i,j}^v = a_{1,i,j}^{\ol v}/a_{0,i,j}^{\ol v}$ for each $i$ and $j$.
As with the $\GL_2$ case, $(a_{1,1,j}^v/a_{0,1,j}^v,\ldots,a_{1,d,j}^v/a_{0,d,j}^v)$ equals the image of 
\[ \rho_v \mod{(\pi,X_1,\ldots,X_j^2,\ldots,X_d)} \]
under the composition:
\[ 
\begin{tikzcd} 
\Def_{\ol\rho_v}^{\ord}(k[\epsilon]) \ar[r,"\cong"] & H^1(\Gamma_{F_v},\b^0) \ar[r] & H^1(\Gamma_{F_v},\b^0/\n) \ar[d,"\cong"] \\
k^d & \Hom(\Z_p^\times,k^d) \ar[l,"ev_{1+p}"'] & \Hom(\Gamma_{F_v}^{\ab},\t^0) \ar[l,"\text{res + CFT}"']
\end{tikzcd}
\]
The image of $\rho_{\ol v}$ is equal to $-w_0(a_{1,1,j}^{\ol v}/a_{0,1,j}^{\ol v},\ldots,a_{1,d,j}^{\ol v}/a_{0,d,j}^{\ol v})$ because of the use of $-w_0$ in the definition of $f_{\ol v,i}$.
Thus the infinitesimal weights of $\rho_v$ and $\rho_{\ol v}$ are parallel.
\end{proof}

\subsection{Higher dimensional weight space}

In preparation for more general number fields, namely CM number fields, we will say another word about weight space.
Just as before, we define
\[ \O_{F,p}^\times = \prod_{v\mid p}\O_{F_v}^\times, \]
and $\W$ to be the rigid space such that for each rigid space $X$,
\[\Hom_{\rig}(X,\W) = \Hom(\O_{F,p}^\times,\O_X(X)).\]
A $\ol\Q_p$-point of weight space is once again given by a character
\[ \chi:\O_{F,p}^\times \to \ol\Q_p^\times. \]
As before, $\W$ is a finite union of open $[F:\Q]$-balls.

\begin{defn}
    A weight $\chi:\O_{F,p}^\times \to \ol\Q_p^\times$ is called \textbf{locally algebraic} if it has the form
\[ \chi = \varepsilon\cdot\prod_{\sigma:F\to\ol\Q_p}\sigma^{n_\sigma} \]
for some integers $n_\sigma$, and finite order character $\varepsilon$.
Here, the product is over all embeddings $\sigma:F\to\ol\Q_p$.
$\chi$ is \textbf{algebraic} if $\varepsilon$ is trivial.
\end{defn}

Suppose that $F/\Q$ is a Galois CM extension with maximal totally real subextension $F^+$
Suppose also that each prime $v\mid p$ of $F^+$ splits in $F$, $v = w\ol w$.
In this case, each embedding $\sigma:F\to \ol\Q_p$ which factors through $F_w$ will have a ``conjugate'' embedding $\ol\sigma:F \to \ol\Q_p$ which factors through $F_{\ol w}$, defined by $\ol\sigma = \sigma\circ c$, where $c \in \Gal(F/\Q)$ denotes complex conjugation.

\begin{defn}
    We say a locally algebraic weight $\chi$ is \textbf{locally parallel} if $n_\sigma = n_{\ol\sigma}$ for each $\sigma:F\to\ol\Q_p$ and \textbf{parallel} if $\varepsilon$ is trivial.
\end{defn}
As before, we let $\W^{\pr}$ denote the subspace of locally parallel weights, $\W_0$ the space of parallel weights, and $\W_\varepsilon$ the translate of weights in $\W_0$ by the finite character $\varepsilon$.

\begin{prop}
    Let $F/\Q$ be CM, with $F^+ = F\cap F^c$.
    Let $m = [F^+:\Q]$.
    Then the closure of $\W_\varepsilon$ in either the Zariski or rigid topology is $m$-dimensional.
    The closure of $\W^{\pr}$ in the Zariski topology is $>m$ dimensional, but the closure of $\W^{\pr}$ in the rigid topology is $m$ dimensional.
\end{prop}

\begin{proof}
    The proof that the closure of $\W_0$ in either topology is $m$-dimensional follows the proof of \Cref{gl2-closed-lemma}, using the norm map from $F$ to $F^+$.
    The result for $\W_\varepsilon$ for any finite character $\varepsilon$ follows by translation.
    
    Infinitely many $\varepsilon$'s implies infinitely many disjoint $\W_\varepsilon$'s in $\W^{\pr}$, which implies $\W^{\pr}$ has Zariski closure of dimension strictly larger than $m$.
    As in the proof of \Cref{gl2-closed}, $p$-adic logarithms cut out a closed rigid analytic $m$-dimensional subspace of $\W$ containing $\W^{\pr}$.
\end{proof}

\subsection{$\tau$-labeled Hodge-Tate cocharacters}
    Suppose for this section that $F/\Q_p$ is a finite extension, and let $I = \{\tau:F\hookrightarrow \ol\Q_p\}$.
    Consider a Hodge-Tate representation $\rho:\Gamma_F \to \GL(V)$ acting on a finite dimensional $\ol\Q_p$-vector space $V$.
    Then $D_{HT}(V) := (V\otimes_{\Q_p}B_{HT})^{\Gamma_F}$ is naturally a graded $F\otimes_{\Q_p}\ol\Q_p$-module.
    In the case $F = \Q_p$ that we considered in previous sections, $D_{HT}(V)$ is just a graded $\ol\Q_p$-vector space, and the nonzero graded pieces encode the Hodge-Tate weights.
    For a general $F$, $D_{HT}(V)$ is a graded module over
    \[ F\otimes_{\Q_p}\ol\Q_p \cong \prod_{\tau \in I}\ol\Q_p, \]
    so $D_{HT}(V)$ decomposes into a collection of graded $\ol\Q_p$-vector spaces, indexed by $I$.
    We can then define $\tau$-labelled Hodge-Tate weights using this collection of $\ol\Q_p$-vector spaces.
    
    For a general reductive group $G$,a representation $\rho:\Gamma_F \to G(\ol\Q_p)$, and $\tau:F\hookrightarrow\ol\Q_p$, we obtain the \textbf{$\tau$-labeled Hodge-Tate cocharacter} $\mu_{HT,\tau}$ using Tannakian formalism as before.

\subsection{Generalizing to CM number fields}

Suppose that $F/\Q$ is a Galois CM extension (the Galois assumption is not necessary, but provides uniformity across primes above $p$, which simplifies the book keeping), $F^+ = F\cap F^c$ $\zeta_p \not\in F$, and each place $v\mid p$ of $F^+$ is split in $F$ as $v = w\ol w$.
Let $m = [F^+:\Q]$, and $f = [F_w:\Q_p]$ for any place $w$ of $F$ dividing $p$.
Retain the assumptions and notation for $G$ from previous sections.

\begin{defn}
    Consider a representation $\rho:\Gamma_F \to G(\ol\Q_p)$.
    For every $v\mid p$, $v = w\ol w$, and every $\tau:F_w \hookrightarrow \ol\Q_p$ there is a conjugate embedding $\ol\tau:F_{\ol w} \hookrightarrow\ol\Q_p$.
    Say that $\rho$ has \textbf{parallel Hodge-Tate weights} if there is a central weight cocharacter $\omega$ of $G$ such that for every $v,w,\tau$,
    \[ [\mu_{HT,\tau}^w] = \omega - [\mu_{HT,\ol\tau}^{\ol w}]. \]
\end{defn}

Let $\Delta(G,B,T) = \{\alpha_1,\ldots,\alpha_d\}$ be the simple roots with respect to $B$ (for which each $\rho_w$ is nearly ordinary).
For $w\mid p$, we define characters $\chi_{w,i} = \alpha_i\circ \rho_w$ as before.
The earlier results of this section generalize to give the following.

\begin{prop}\label{full-geom-prop}
    Suppose that $\rho$ is nearly ordinary, non-split, and de Rham at each $w\mid p$, and assume that $\rho$ has parallel Hodge-Tate weights.
    Then $\chi_{w,i}$ and $-w_0\chi_{\ol w,i}$ have the same $\tau$-labeled Hodge-Tate weights for each $i$, $w\mid p$ and $\tau:F_w \hookrightarrow\ol\Q_p$.
\end{prop}

\subsection{Passage to infinitesimal weights for CM fields} The universal deformations constructed in \Cref{rammethod} have the form
\[ \rho: \Gamma_F \to G(\O\lb X_1,\ldots,X_{md}\rb). \]
For $w\mid p$, we define characters $\chi_{w,i} = \alpha_i\circ \rho_w$ as before.
Fixing $i$ and ranging over $w$ determines maps $\gamma_i:\D^{md} \to \W$ induced by
\[ \prod_{\substack{v\mid p \\ v = w\ol w}}\chi_{w,i}\times(-w_0\chi_{\ol w,i}):\O_{F,p}^\times \to \O\lb X_1,\ldots,X_{md}\rb^\times \hookrightarrow \O_{\D^{md}}(\D^{md})^\times. \]

\begin{defn}
For a cocycle $\varphi \in H^1(\Gamma_{F_v},\b^0)$ the \textbf{infinitesimal weight} of $\varphi$ is given by applying to $\varphi$ the composition of maps
\[
\begin{tikzcd}
H^1(\Gamma_{F_v},\b^0) \ar[r] \ar[dr,"\beta",dashed] & H^1(\Gamma_{F_v},\b^0/\n) \ar[r,"\cong"] & \Hom(\Gamma_{F_v}^{\ab},\t^0) \ar[d,"\text{res + CFT}"] \\
& k^{fd} & \Hom(\O_{F_v}^\times,k^d) \ar[l,"ev_{\alpha_1,\ldots,\alpha_f}"']
\end{tikzcd}
\]
where $\alpha_1,\ldots,\alpha_f$ denotes a $\Z_p$-basis for $\O_{F_v}^\times/$torsion.
If $v\mid p$, $v = w\ol w$, $\varphi_w \in H^1(\Gamma_{F_w},\b^0)$, and $\varphi_{\ol w} \in H^1(\Gamma_{F_{\ol w}},\b^0)$, then say that $\varphi_w$ and $\varphi_{\ol w}$ have \textbf{parallel infinitesimal weight} if $\beta(\varphi_w) = -w_0\beta(\varphi_{\ol w})$.
\end{defn}

\begin{prop}\label{full-passage}
    One can associate a collection of infinitesimal weights to $\rho$ at each $w\mid p$ such that either
    \begin{enumerate}
        \item the infinitesimal weights of $\rho_w$ and $\rho_{\ol w}$ are parallel for each $w \mid p$, or 
        \item There is an $1\leq i \leq d$ such that $\gamma_i^{-1}(\W_\varepsilon)$ has codimension $\geq 1$ for each $\varepsilon$, and $\gamma_i^{-1}(W^{\pr})$ has codimension $\geq 1$ in the rigid topology.
    \end{enumerate}
\end{prop}

\begin{proof}
Let $\alpha_1,\ldots,\alpha_f$ denote a basis for $\O_{F_v}^\times/$torsion.
Define for each $i=1,\ldots,d$ and $j=1,\ldots,f$
\[ f_{w,i,j} := \chi_{w,i}(\alpha_j),\qquad\text{and}\qquad f_{\ol w,i,j} := -w_0\chi_{\ol w,i}(\alpha_j).\]
For each $i$ and $j$ and each root of unity $\zeta$, consider the series
\[ g_{w,i,j,\zeta} := \frac{f_{w,i,j}}{f_{\ol w,i,j}} - \zeta \in \ol\O\lb X_1,\ldots,X_{md}\rb. \]
If for some $i,j,w$, $f_{w,i,j}/f_{\ol w,i,j}$ is not a constant root of unity, then for each $\zeta$, $g_{w,i,j,\zeta}$ is 0 on a codimension $\geq 1$ set of specializations.
We deduce that $\gamma_i^{-1}(W_\varepsilon)$ is codimension $\geq 1$ for every finite character $\varepsilon$, and (2) holds.

Otherwise, for every $w,i,j$ we have $f_{w,i,j} = \zeta_{w,i,j}f_{\ol w,i,j}$ for some root of unity $\zeta_{w,i,j}$.
We have maps $\O\lb X_1,\ldots,X_{md}\rb \twoheadrightarrow k[\epsilon]$ defined for $1\leq n \leq md$ by the reductions
\[ \O\lb X_1,\ldots,X_{md}\rb \twoheadrightarrow \O\lb X_1,\ldots,X_{md}\rb/(\pi,X_1,\ldots,X_n^2,\ldots,X_{md}) \cong k[\epsilon], \]
and for each $i,j,n,w$ we write
\[ f_{w,i,j} \equiv a_{0,i,j,n}^w + a_{1,i,j,n}^w\epsilon \mod{(\pi,X_1,\ldots,X_n^2,\ldots,X_{md})}. \]
Then since $f_{w,i,j}$ and $f_{\ol w,i,j}$ differ only by a constant, we have
\[ \frac{a_{1,i,j,n}^{w}}{a_{0,i,j,n}^{w}} = \frac{a_{1,i,j,n}^{\ol w}}{a_{0,i,j,n}^{\ol w}}. \]
Further, the tuple
\[ \left(\frac{a_{1,i,j,n}^{w}}{a_{0,i,j,n}^{w}}\right)_{\substack{i=1,\ldots,d \\ j=1,\ldots,f}} \]
is the image of $\rho_w \mod{(\pi,X_1,\ldots,X_n^2,\ldots,X_{md})}$ under the composition
\[ 
\begin{tikzcd} 
\Def_{\ol\rho_v}^{\ord}(k[\epsilon]) \ar[r,"\cong"] & H^1(\Gamma_{F_v},\b^0) \ar[r] & H^1(\Gamma_{F_v},\b^0/\n) \ar[d,"\cong"] \\
k^{fd} & \Hom(\O_{F_v}^\times,k^d) \ar[l,"ev_{\alpha_1,\ldots,\alpha_f}"'] & \Hom(\Gamma_{F_v}^{\ab},\t^0) \ar[l,"\text{res + CFT}"']
\end{tikzcd}
\]
and
\[ -w_0\left(\frac{a_{1,i,j,n}^{\ol w}}{a_{0,i,j,n}^{\ol w}}\right)_{\substack{i=1,\ldots,d \\ j=1,\ldots,f}} \]
is the image of $\rho_{\ol w}$.
Thus $\rho$ has parallel infinitesimal weights at each pair $w,\ol w\mid p$.
\end{proof}
\section{Sparsity of automorphic points}\label{sparse}

\subsection{Introduction} 

Suppose that $\ol\rho$ satisfies the hypotheses of \Cref{large-image} and that $\rho$ is the universal deformation produced by \Cref{large-image}.
To ease notation, we write $\Sigma$ for the set of places denoted $\Sigma \cup Q$ in the conclusion of \Cref{large-image}.
We have an infinitesimal weight map:
\begin{equation}\label{weight-map}
\begin{tikzcd}
\Def_{\ol\rho}^{\P}(k[\epsilon]) \ar[r,"\cong"] & H^1_{\P}(\Gamma_{F,\Sigma},\g^0) \ar[r,"\res"] & \prod_{w\mid p}H^1(\Gamma_{F_w},\b^0) \ar[r] & \prod_{w\mid p}k^{fd} \cong k^{[F:\Q]d}
\end{tikzcd}
\end{equation}

If we assume that $\ol\rho$ is non-split at each $w\mid p$, then $\rho$ will be non-split for all specializations off of a codimension $\geq 1$ subspace.
If any non-split specialization is geometric with parallel weight, \Cref{full-geom-prop} applies.
Then by \Cref{full-passage}, either
\begin{enumerate}
    \item the set of specializations of $\rho$ which are geometric with parallel weight has codimension $\geq 1$, or
    \item The image of (\ref{weight-map}) lands in the \textbf{parallel weight subspace} of $\prod_{w\mid p}k^{fd}$ where the coordinates for $w$ and $\ol w$ differ by $-w_0$.
\end{enumerate}

Suppose that (2) holds.
Then the image of (\ref{weight-map}) is contained in the (proper) subspace of parallel infinitesimal weights.
Following the ideas of \cite{cm}, we will show how to add an additional prime of ramification to $\Sigma$ to produce a new universal deformation satisfying the conclusions of \Cref{large-image}, but for which  (2) does not hold.
Thus (1) will hold, and we will have produced a deformation space for $\ol\rho$ with a sparsity of geometric parallel weight (and conjecturally automorphic) points.

In particular, assuming the image of (\ref{weight-map}) is contained in any proper subspace $U$, we will show how to add a prime to $\Sigma$ so that the image of (\ref{weight-map}) is not contained in $U$.
We begin with the following Galois cohomology result, which is modeled after Proposition 10 in \cite{r-def}.
Compare also to Lemma 7.7 in \cite{cm}.

\begin{prop}
    Let $\psi \in H^1_{\P}(\Gamma_{F,\Sigma},\g^0)$ be a non-zero Selmer class.
    Let $T$ denote the \v Cebotarev class determined by \Cref{ceb-lemma} for $\psi$.
    Then the restriction map
    \[ \theta:H^1(\Gamma_{F,\Sigma\cup T},\g^0) \to \bigoplus_{v\in\Sigma}H^1(\Gamma_v,\g^0)  \]
    is surjective.
\end{prop}

\begin{proof}
    For any set of primes $W$, define $P_W^1(M) = \prod_{v\in W}'H^1(\Gamma_{F_v},M)$, where the product is restricted with respect to the unramified cocycles (the image of the inflation map $H^1(\Gamma_{F_v}/I_{F_v},M) \to H^1(\Gamma_{F_v},M)$).
    Then we have restriction maps
    \[ \tilde\theta: H^1(\Gamma_{F,\Sigma\cup T},\g^0) \to P^1_{\Sigma\cup T}(\g^0), \]
    \[ \tilde\vartheta: H^1(\Gamma_{F,\Sigma\cup T},\g^0(1)) \to P^1_{\Sigma\cup T}(\g^0(1)). \]
    Notice that $P^1_{\Sigma\cup T}(\g^0)$ and $P^1_{\Sigma\cup T}(\g^0(1))$ are dual via local duality.
    By the results of global duality, $\image\tilde\theta$ is the annihilator of $\image\tilde\vartheta$ under the local duality pairing.
    Then the annihilator of $P_T^1(\g^0) + \image\tilde\theta$ in $P_{\Sigma\cup T}^1(\g^0(1))$ is the intersection of the annihilators of $P_T^1(\g^0)$ and $\image\tilde\theta$, which we can describe as
    \[ \{\tilde\vartheta(\phi) : \phi \in H^1(\Gamma_{F,\Sigma\cup T},\g^0(1)),\ \phi|_{\Gamma_{F_v}} = 0\ \forall v \in T\}. \]
    I claim that any $\phi$ is trivial.
    Assuming the claim for a moment, we conclude that
    \[ P_T^1(\g^0) + \image\tilde\theta = P_{\Sigma\cup T}^1(\g^0), \]
    so that (all cohomology applied to the module $\g^0$)
    \begin{align*}
        \coker\theta &= \coker\left(H^1(\Gamma_{F,\Sigma\cup T}) \xrightarrow{\tilde\theta} P^1_{\Sigma\cup T} \twoheadrightarrow \frac{P_{\Sigma\cup T}^1}{P_T^1} \cong \bigoplus_{v\in \Sigma} H^1(\Gamma_{F_v})\right) \\
        &= \frac{P^1_{\Sigma\cup T}}{P_T^1 + \image\tilde\theta} = 0,
    \end{align*}
    as desired.
    
    In remains to prove the claim.
    Suppose that there exists a non-zero $\phi \in H^1(\Gamma_{F,\Sigma\cup T},\g^0(1))$ such that $\phi|_{\Gamma_{F_v}} = 0$ for all $v \in T$.
    We can re-run \Cref{ceb-lemma} using $\psi$ and $\phi$ (see the remark after the proof of \Cref{large-image}), to produce a \v Cebotarev class $\tilde T \subseteq T$, and for any $v \in \tilde T$, $\phi|_{\Gamma_{F_v}} \not\in L_v^{\Ram,\perp}$, in particular $\phi|_{\Gamma_{F_v}} \neq 0$, contradicting the choice of $\phi$.
\end{proof}

\begin{cor}
    There is a surjection $\beta:H^1(\Gamma_{F,\Sigma\cup T},\g^0) \to \prod_{v\mid p}k^{fd}$ onto infinitesimal weight space.
\end{cor}
\begin{proof}
    The way that $H^1(\Gamma_{F,\Sigma\cup T},\g^0)$ maps to infinitesimal weight space is via the composition:
    \begin{align*}
        \beta:H^1(\Gamma_{F,\Sigma\cup T},\g^0) &\xrightarrow{\theta} \bigoplus_{v\in\Sigma}H^1(\Gamma_v,\g^0)
        \cong \bigoplus_{v\in\Sigma}\left(H^1(\Gamma_v,\g^0)/L_v\oplus L_v\right) \\
        &\twoheadrightarrow \bigoplus_{v\in\Sigma}L_v \\
        &\twoheadrightarrow \bigoplus_{v\mid p}H^1(\Gamma_{F_v},\b^0) \\
        &\xrightarrow{(*)} \bigoplus_{v\mid p}H^1(\Gamma_{F_v},\b^0/\n)
        \cong \bigoplus_{v\mid p}\Hom(\Gamma_{F_v}^{ab},\t^0) \\
        &\twoheadrightarrow \bigoplus_{v\mid p}\Hom(\O_{F_v}^\times,k^d)
        \twoheadrightarrow \bigoplus_{v\mid p}k^{fd}.
    \end{align*}
    $\theta$ is surjective by the proposition, and the isomorphism of the top line comes from choosing a splitting.
    The map $(*)$ is surjective by assumption (REG*), since part of the long exact sequence associated to 
    \[ 0 \to \n \to \b^0 \to \b^0/\n \to 0, \]
    is given by
    \[ H^1(\Gamma_{F_v},\b^0) \to H^1(\Gamma_{F_v},\b^0/\n) \to H^2(\Gamma_{F_v},\n), \]
    and by local duality,
    \[ H^2(\Gamma_{F_v},\n) \cong H^0(\Gamma_{F_v},\g/\b(1)) = 0. \]
\end{proof}

Recall that we are assuming that the image of $H^1_\P(\Gamma_{F,\Sigma},\g^0)$ under $\beta$ is contained in a proper subspace $U$.
Now choose primes $y_1,\ldots,y_k$ such that $\beta(H^1(\Gamma_{F,\Sigma\cup \{y_1,\ldots,y_k\}},\g^0)) \not\subset U$.

\begin{lem}
    We may assume $k=1$, i.e.~there exists a single prime $y \in \tilde T$ such that the image of $H^1(\Gamma_{F,\Sigma\cup\{y\}},\g^0)$ under $\beta$ is not contained in $U$.
\end{lem}

\begin{proof}
    Recall that as a consequence of \Cref{ceb-lemma}, if we add the appropriate local conditions at $y_i$'s to our deformation problem, the dual Selmer group still vanishes.
    Since dual Selmer vanishes, we also have the weaker statement that
    \[ \Sha^1_{\Sigma'}(\g^0(1)) := \ker\left(H^1(\Gamma_{F,\Sigma'},\g^0(1)) \to \prod_{v\in\Sigma'}H^1(\Gamma_{F_v},\g^0(1))\right) \]
    vanishes for $\Sigma' = \Sigma$, $\Sigma\cup\{y_i\}$, $\Sigma\cup\{y_1,\ldots,y_k\},$ etc.
    Using Wiles's formula, the local Euler characteristic formula, and local duality, we deduce
    \begin{align*}
        h^1(\Gamma_{F,\Sigma'},\g^0) - h^1(\Gamma_{F,\Sigma},\g^0)
        &= \sum_{y \in \Sigma'\ssm \Sigma} \left(h^1(\Gamma_{F_y},\g^0) - h^0(\Gamma_{F_y},\g^0)\right) \\
        &= \sum_{y \in \Sigma'\ssm \Sigma} h^2(\Gamma_{F_y},\g^0) \\
        &= \sum_{y \in \Sigma'\ssm \Sigma} h^0(\Gamma_{F_y},\g^0(1)).
    \end{align*}
    Applying this to each $\Sigma' = \Sigma\cup\{y_i\}$ and to $\Sigma' = \Sigma\cup\{y_1,\ldots,y_k\}$, dimension considerations imply that the inflation maps induce an isomorphism
    \[ \bigoplus_{i=1}^k\frac{H^1(\Gamma_{F,\Sigma\cup\{y_i\}},\g^0)}{H^1(\Gamma_{F,\Sigma},\g^0)} \cong \frac{H^1(\Gamma_{F,\Sigma\cup\{y_1,\ldots,y_k\}},\g^0)}{H^1(\Gamma_{F,\Sigma},\g^0)}. \]
    So if $\phi \in H^1(\Gamma_{F,\Sigma\cup\{y_1,\ldots,y_k\}},\g^0)$ such that $\beta(\phi) \not\in U$, then there is some $i$th component of $\phi$ on the left hand side of this isomorphism whose image under $\beta$ is not in $U$.
    Let $y = y_i$.
    Then the image of $H^1(\Gamma_{F,\Sigma\cup \{y\}},\g^0)$ under $\beta$ is not contained in $U$.
\end{proof}

Fix $y \in \tilde T$ such that the image of $H^1(\Gamma_{F,\Sigma\cup \{y\}},\g^0)$ under $\beta$ is not contained in $U$, and let $\Sigma' = \Sigma\cup\{y\}$.
By the computation above,
\[ h^1(\Gamma_{F,\Sigma'},\g^0) - h^1(\Gamma_{F,\Sigma},\g^0) = h^0(\Gamma_{F_y},\g^0(1)) = 1, \]
since $\ol\rho|_{\Gamma_{F_y}}$ is of Ramakrishna type (there is a unique root acting as the cyclotomic character).
% Let \[ \psi' \in H^1(\Gamma_{F,\Sigma'},\g^0)\ssm H^1(\Gamma_{F,\Sigma},\g^0) \]
% (so $\psi'$ has image outside $U$).
Let $\P_y$ be the local deformation condition with tangent space $L_y = L_y^{\unr} \cap L_y^{\Ram}$, and let $\P' = \P \cup \{\P_y\}$.
Then by \Cref{ceb-lemma}, $\psi \not\in H^1_{\P'}(\Gamma_{F,\Sigma'},\g^0)$.

\begin{prop}
    There exists a cocycle in $\tilde\psi \in H^1_{\P'}(\Gamma_{F,\Sigma'},\g^0)$ such that $\beta(\tilde\psi) \not\in U$.
\end{prop}

\begin{proof}
    Consider the following commutative diagram whose rows are exact sequences (all cohomology is applied to the module $\g^0$):
    \[
    \begin{tikzcd}
    1 \ar[r] & H^1_{\P'} \ar[r] & H^1(\Gamma_{F,\Sigma'}) \ar[r,"\res"] \ar[rd,dashed,"\Phi"] & \bigoplus_{v\in \Sigma'}H^1(\Gamma_{F_v})/L_v \ar[r] \ar[d,two heads] & 1 \\
    1 \ar[r] & H^1_{\P} \ar[r] & H^1(\Gamma_{F,\Sigma}) \ar[r,"\res"] \ar[u,hook,"\inf"] & \bigoplus_{v\in \Sigma}H^1(\Gamma_{F_v})/L_v \ar[r] & 1
    \end{tikzcd}
    \]
    By the numerology,
    \[ \dim_kH^1_{\P'} = \dim_kH^1_{\P} = md, \]
    and the difference in dimensions between the right hand terms is 1.
    Notice that $H_{\P}^1$ injects into $H^1(\Gamma_{F,\Sigma'})$ and $H_{\P}^1 \subset \ker\Phi$.
    But $\ker\Phi$ is $(md+1)$-dimensional, so fix $\psi' \in \ker\Phi$ so that $\psi'$ and $H^1_\P$ span $\ker\Phi$.
    Notice that cocycles which are not in $\ker\Phi$ can have no effect on infinitesimal weights, and $\beta(H^1_\P) \subset U$.
    Since $\beta(H^1(\Gamma_{F,\Sigma'})) \not\subset U$, we must have $\beta(\psi') \not\in U$.
    Then for dimension reasons, $H_{\P'}^1$ is contained in the span of $H_{\P}^1$ and $\psi$ in $H^1(\Gamma_{F,\Sigma'})$.
    Since $\psi \not\in H_{\P'}^1$, we conclude that some element $\tilde\psi \in H_{\P'}^1$ has non-zero $\psi'$-component.
    Since all elements of $H^1_{\P}$ are assumed to have image in $U$, $\beta(\tilde\psi) \not\in U$.
\end{proof}

We can now assemble what we've done into our main theorem.

\begin{thm}\label{final-theorem}
    Let $k$ denote a finite field of characteristic $p$, $G/W(k)$ a split connected reductive group, $F$ a Galois CM number field satisfying $[F(\mu_p):F] = p-1$ with maximal totally real subextension $F^+$ for which every prime $v\mid p$ of $F^+$ splits in $F$, and $\Sigma$ a finite set of primes of $F$ containing the primes above $p$.
    Assume $\ol\rho:\Gamma_{F,\Sigma} \to G(k)$ is a continuous representation satisfying:
    \begin{enumerate}
        \item There is a subfield $k' \subset k$ such that 
        \[ \ol{(G^{\der})}^{\simpc}(k') \subset \ol\rho(\Gamma_F) \subset Z_G(k)\cdot G(k'). \]
        \item $p-1$ is greater than the maximum of $8\# Z_{(G^{\der})^{\simpc}}$ and
        \[
        \begin{cases}
            (h-1)\#Z_{(G^{\der})^{\simpc}} & \text{if $\#Z_{(G^{\der})^{\simpc}}$ is even; or} \\
            (2h-2)\#Z_{(G^{\der})^{\simpc}} & \text{if $\#Z_{(G^{\der})^{\simpc}}$ is odd.}
        \end{cases}
        \]
        \item For all places $v \in \Sigma$ not dividing $p\cdot\infty$, $\ol\rho|_{\Gamma_{F_v}}$ satisfies a liftable local deformation condition $\P_v$ with tangent space of dimension $h^0(\Gamma_{F_v},\g^0)$.
        \item For all places $v\mid p$, $\ol\rho|_{\Gamma_{F_v}}$ is nearly ordinary and non-split, such that $\alpha\circ\ol\rho|_{\Gamma_{F_v}}$ is not trivial or $\ol\kappa$ for each simple root $\alpha \in \Delta$.
    \end{enumerate}
    Then letting $r = \frac{[F:\Q]}2\dim\t^0$, there exists a representation
    \[ \rho : \Gamma_F \to G(W(k)\lb X_1,\ldots,X_r\rb) \] lifting $\ol\rho$ and such that $\rho$ is almost everywhere unramified, and nearly ordinary at all $v\mid p$.
    Further, the set of $\ol\Q_p$-points of $\rho$ which are geometric and have parallel Hodge-Tate weights has positive codimension in $\D^r$
\end{thm}

\begin{proof}
    Notice that since $\ol\rho$ is non-split, the only possible invariants of a twist of $\g/\b$ live in negative simple root spaces.
    Then assuming $\alpha\circ\ol\rho
    _{\Gamma_{F_v}}$ is not trivial or cyclotomic, ensures $\g/\b$ and $\g/\b(1)$ have no invariants, i.e.~(REG) and (REG*) hold.
    Now use \Cref{large-image} to produce a universal deformation $\rho$. 
    If the image of \Cref{weight-map} lands in the space of parallel infinitesimal weights, then use the results of this section to create a new deformation problem $\P'$ with universal deformation $\rho'$ which still satisfies the conclusions of \Cref{large-image}, but for which the image of \Cref{weight-map} does not land entirely in the space of parallel infinitesimal weights.
    
    Thus we may assume that $\rho$ has non-parallel infinitesimal weights.
    By \Cref{full-passage}, \Cref{full-geom-prop}, and the fact that $\rho$ is non-split at all $w\mid p$ off a codimension $\geq 1$ subvariety, the space of specializations of $\rho$ which are geometric and have parallel Hodge-Tate weights has codimension $\geq 1$ in the rigid topology.
\end{proof}

\begin{rmk}
    To make the result more ``automorphic,'' perhaps we should have worked with $L$-groups instead of just reductive groups.
    Working in this generality shouldn't produce any new complications.
    The deformation theory results that we have used have analogues in \cite[\S9]{pat} which apply to $L$-groups, and the role of $\g$ in cohomology is filled by the Lie algebra of $G^\vee := (^LG)^\circ$.
\end{rmk}

\begin{rmk}
    The ``non-split'' condition was added to ensure that we only had to avoid a proper subspace of infinitesimal weight space, and not $\#W_G$ proper subspaces, which most likely span all of weight space.
    It is possible that we could remove the non-splitness condition using the techniques of \cite{fkp}.
    I might revisit this at a later time.
\end{rmk}
\section{An example}\label{s-example}

We can provide examples where the hypotheses of \Cref{mainthm} are satisfied by using the potential solution of the inverse Galois problem with local conditions.
The following is Proposition 3.2 in \cite{even2}, and is due to Moret-Bailly \cite{mb}.

\begin{prop}\label{p-igp}
    Let $G$ be a finite group, let $K/\Q$ be a finite extension, and $S$ a finite set of places of $K$.
    Let $E/K$ be an auxilary finite extension of number fields.
    For each finite place $v \in S$, let $H_v/K_v$ be a finite Galois extension together with a fixed inclusion $\phi_v:\Gal(H_v/K_v) \to G$ with image $D_v$.
    For each real infinite place $v \in S$, let $c_v \in G$ be an element of order dividing 2.
    There exists a number field $F/K$ and a finite Galois extension of number fields $L/F$ with the following properties:
    \begin{enumerate}
        \item There is an isomorphism $\Gal(L/F) = G$.
        \item $L/K$ is linearly disjoint from $E/K$.
        \item All places in $S$ split completely in $F$.
        \item For all finite places $w$ of $F$ above $v \in S$, the local extension $L_w/F_w$ is equal to $H_v/K_v$.
        Moreover there is a commutative diagram:
        \[
        \begin{tikzcd}
        \Gal(L_w/F_w) \ar[r] \ar[d,equals] & D_w \subset G \ar[d,equals] \\
        \Gal(H_v/K_v) \ar[r,"\phi_v"] & D_v \subset G
        \end{tikzcd}
        \]
        \item For all real places $w\mid \infty $ of $F$ above $v \in S$, complex conjugation $c_w \in G$ is conjugate to $c_v$.
    \end{enumerate}
\end{prop}

\begin{cor}\label{c-example}
    For any split connected reductive group $G/W(k)$ for which $k$ is a finite field of charicteristic $p \gg 0$, there exists a Galois CM field $F$, a set of primes $\Sigma$, and a Galois representation $\ol\rho:\Gamma_{F,\Sigma} \to G(k)$ satisfying the conditions of \Cref{mainthm}.
\end{cor}

\begin{proof}
    In the notation of \Cref{p-igp}, let $G = G(k)$, and $K = \Q$, $S = \{p,\infty\}$.
    We will choose to avoid a quadratic imaginary extension $E/\Q$ in which $p$ splits.
    Let $c_\infty \in G(k)$ be any order 2 element.
    To apply \Cref{p-igp} it remains to specify an extension $H_p/\Q_p$.
    
    Fix a pinning of $G$ and let $r:\SL_2 \to G$ denote the principle $\SL_2$ with respect to this pinning (see \cite{gross}).
    Let $\varrho:\Gamma_{\Q_p} \to \SL_2(k)$ be a non-split extension of $\F_p$ by $\F_p(r)$ for $r \neq 0,1$.
    Such an extension exists, since 
    \begin{itemize}
        \item $h^0(\Gamma_{\Q_p},\F_p(r)) = 0$ for $r\neq 0$,
        \item $h^2(\Gamma_{\Q_p},\F_p(r)) = h^0(\Gamma_{\Q_p},\F_p(1-r)) = 0$ for $r\neq 1$ (by local duality), and
        \item $h^1(\Gamma_{\Q_p},\F_p(r)) = h^0(\Gamma_{\Q_p},\F_p(r)) + h^2(\Gamma_{\Q_p},\F_p(r)) + 1 = 1$ for $r\neq 0,1$ (by the local Euler characteristic formula).
    \end{itemize}
    Now let $\phi_p:\Gamma_{\Q_p} \to G(k)$ be given by $r \circ \varrho$, and let $H_p$ denote the fixed field of $\ker\phi_p$, so that $\phi_p$ induces an inclusion $\phi_p:\Gal(H_p/\Q_p) \to G(k)$.
    
    Then by \Cref{p-igp}, there exists
    \begin{itemize}
        \item a (totally real) number field $F^+$ such that all primes above $p$ split,
        \item an extension $L/F^+$ such that $L/\Q$ is linearly disjoint from $E/\Q$, and
        \item an isomorphism $\Gal(L/F^+) \cong G(k)$
    \end{itemize}
    satisfying the conclusions of \Cref{p-igp}.
    Let $F = E\cdot F^+$, and let 
    \[ \ol\rho: \Gamma_F \twoheadrightarrow \Gal(LE/F) \cong \Gal(L/F^+) \cong G(k).\]
    After a possibly further base change we may assume that $\ol\rho$ is unramified away from $p$.
    Let $\Sigma$ denote the set of primes above $p$ ($p$ is totally split in $F$), and note that $\ol\rho$ factors through $\Gamma_{F,\Sigma}$.
    Since $\ol\rho$ is surjective, it satisfies (1) of \Cref{mainthm}.
    Since $\Sigma$ contains no primes $v\nmid p\cdot \infty$, (2) of \Cref{mainthm} is vacuously true.
    Finally, for each $v\mid p$, $\ol\rho|_{\Gamma_{F_v}} = \phi_p$.
    Then $\ol\rho|_{\Gamma_{F_v}}$ is ordinary for the Borel $B$ of our fixed pinning of $G$, and for any simple root $\alpha$, $\alpha\circ \ol\rho|_{\Gamma_{F_v}} = \ol\kappa^r$ (and $r\neq 0,1$).
    Further, since the image of $\varrho$ contains a non-trivial unipotent element, the image of $\ol\rho|_{\Gamma_{F_v}}$ contains a regular unipotent element, in particular $\ol\rho|_{\Gamma_{F_v}}$ is non-split.
    Thus $\ol\rho$ satisfies (3) of \Cref{mainthm}.
\end{proof}
\newpage
\bibliographystyle{alpha}
\bibliography{refs}

\begin{thebibliography}{BLGGT14}

\bibitem[All19]{allen}
Patrick~B. Allen.
\newblock On automorphic points in polarized deformation rings.
\newblock {\em Amer. J. Math.}, 141(1):119--167, 2019.

\bibitem[B\"01]{bockle}
Gebhard B\"{o}ckle.
\newblock On the density of modular points in universal deformation spaces.
\newblock {\em Amer. J. Math.}, 123(5):985--1007, 2001.

\bibitem[BG14]{bg}
Kevin Buzzard and Toby Gee.
\newblock The conjectural connections between automorphic representations and
  {G}alois representations.
\newblock In {\em Automorphic forms and {G}alois representations. {V}ol. 1},
  volume 414 of {\em London Math. Soc. Lecture Note Ser.}, pages 135--187.
  Cambridge Univ. Press, Cambridge, 2014.

\bibitem[BK90]{bk}
Spencer Bloch and Kazuya Kato.
\newblock {$L$}-functions and {T}amagawa numbers of motives.
\newblock In {\em The {G}rothendieck {F}estschrift, {V}ol. {I}}, volume~86 of
  {\em Progr. Math.}, pages 333--400. Birkh\"{a}user Boston, Boston, MA, 1990.

\bibitem[BLGGT14]{BLGGT}
Thomas Barnet-Lamb, Toby Gee, David Geraghty, and Richard Taylor.
\newblock Potential automorphy and change of weight.
\newblock {\em Ann. of Math. (2)}, 179(2):501--609, 2014.

\bibitem[Buz04]{buzzard}
Kevin Buzzard.
\newblock On {$p$}-adic families of automorphic forms.
\newblock In {\em Modular curves and abelian varieties}, volume 224 of {\em
  Progr. Math.}, pages 23--44. Birkh\"{a}user, Basel, 2004.

\bibitem[BW00]{bw}
A.~Borel and N.~Wallach.
\newblock {\em Continuous cohomology, discrete subgroups, and representations
  of reductive groups}, volume~67 of {\em Mathematical Surveys and Monographs}.
\newblock American Mathematical Society, Providence, RI, second edition, 2000.

\bibitem[Cal11]{even1}
Frank Calegari.
\newblock Even {G}alois representations and the {F}ontaine-{M}azur conjecture.
\newblock {\em Invent. Math.}, 185(1):1--16, 2011.

\bibitem[Cal12]{even2}
Frank Calegari.
\newblock Even {G}alois representations and the {F}ontaine--{M}azur conjecture.
  {II}.
\newblock {\em J. Amer. Math. Soc.}, 25(2):533--554, 2012.

\bibitem[Car93]{carter}
Roger~W. Carter.
\newblock {\em Finite groups of {L}ie type}.
\newblock Wiley Classics Library. John Wiley \& Sons, Ltd., Chichester, 1993.
\newblock Conjugacy classes and complex characters, Reprint of the 1985
  original, A Wiley-Interscience Publication.

\bibitem[CE09]{ce}
Frank Calegari and Matthew Emerton.
\newblock Bounds for multiplicities of unitary representations of cohomological
  type in spaces of cusp forms.
\newblock {\em Ann. of Math. (2)}, 170(3):1437--1446, 2009.

\bibitem[CG18]{cg}
Frank Calegari and David Geraghty.
\newblock Modularity lifting beyond the {T}aylor-{W}iles method.
\newblock {\em Invent. Math.}, 211(1):297--433, 2018.

\bibitem[CHT08]{CHT}
Laurent Clozel, Michael Harris, and Richard Taylor.
\newblock Automorphy for some {$l$}-adic lifts of automorphic mod {$l$}
  {G}alois representations.
\newblock {\em Publ. Math. Inst. Hautes \'Etudes Sci.}, (108):1--181, 2008.
\newblock With Appendix A, summarizing unpublished work of Russ Mann, and
  Appendix B by Marie-France Vign\'eras.

\bibitem[CM09]{cm}
Frank Calegari and Barry Mazur.
\newblock Nearly ordinary {G}alois deformations over arbitrary number fields.
\newblock {\em J. Inst. Math. Jussieu}, 8(1):99--177, 2009.

\bibitem[dJ95]{de-jong}
A.~J. de~Jong.
\newblock Crystalline {D}ieudonn\'{e} module theory via formal and rigid
  geometry.
\newblock {\em Inst. Hautes \'{E}tudes Sci. Publ. Math.}, (82):5--96 (1996),
  1995.

\bibitem[DMOS82]{dm}
Pierre Deligne, James~S. Milne, Arthur Ogus, and Kuang-yen Shih.
\newblock {\em Hodge cycles, motives, and {S}himura varieties}, volume 900 of
  {\em Lecture Notes in Mathematics}.
\newblock Springer-Verlag, Berlin-New York, 1982.

\bibitem[FKP19a]{fkp2}
Najmuddin {Fakhruddin}, Chandrashekhar {Khare}, and Stefan {Patrikis}.
\newblock {Lifting $G$-irreducible but $\mathrm{GL}_n$-reducible Galois
  representations}.
\newblock {\em arXiv e-prints}, page arXiv:1908.07929, Aug 2019.

\bibitem[FKP19b]{fkp}
Najmuddin {Fakhruddin}, Chandrashekhar {Khare}, and Stefan {Patrikis}.
\newblock {Relative deformation theory and lifting irreducible Galois
  representations}.
\newblock {\em arXiv e-prints}, page arXiv:1904.02374, Apr 2019.

\bibitem[FM95]{fm}
Jean-Marc Fontaine and Barry Mazur.
\newblock Geometric {G}alois representations.
\newblock In {\em Elliptic curves, modular forms, \& {F}ermat's last theorem
  ({H}ong {K}ong, 1993)}, Ser. Number Theory, I, pages 41--78. Int. Press,
  Cambridge, MA, 1995.

\bibitem[GM98]{gm}
Fernando~Q. Gouv\^{e}a and Barry Mazur.
\newblock On the density of modular representations.
\newblock In {\em Computational perspectives on number theory ({C}hicago, {IL},
  1995)}, volume~7 of {\em AMS/IP Stud. Adv. Math.}, pages 127--142. Amer.
  Math. Soc., Providence, RI, 1998.

\bibitem[Gro97]{gross}
Benedict~H. Gross.
\newblock On the motive of {$G$} and the principal homomorphism {${\rm
  SL}_2\to\widehat G$}.
\newblock {\em Asian J. Math.}, 1(1):208--213, 1997.

\bibitem[Har87]{harder}
G.~Harder.
\newblock Eisenstein cohomology of arithmetic groups. {T}he case {${\rm
  GL}_2$}.
\newblock {\em Invent. Math.}, 89(1):37--118, 1987.

\bibitem[HLTT16]{hltt}
Michael Harris, Kai-Wen Lan, Richard Taylor, and Jack Thorne.
\newblock On the rigid cohomology of certain {S}himura varieties.
\newblock {\em Res. Math. Sci.}, 3:Paper No. 37, 308, 2016.

\bibitem[Kna02]{knapp}
Anthony~W. Knapp.
\newblock {\em Lie groups beyond an introduction}, volume 140 of {\em Progress
  in Mathematics}.
\newblock Birkh\"{a}user Boston, Inc., Boston, MA, second edition, 2002.

\bibitem[KW09a]{kw1}
Chandrashekhar Khare and Jean-Pierre Wintenberger.
\newblock Serre's modularity conjecture. {I}.
\newblock {\em Invent. Math.}, 178(3):485--504, 2009.

\bibitem[KW09b]{kw2}
Chandrashekhar Khare and Jean-Pierre Wintenberger.
\newblock Serre's modularity conjecture. {II}.
\newblock {\em Invent. Math.}, 178(3):505--586, 2009.

\bibitem[Loe11]{loeffler}
David Loeffler.
\newblock Density of classical points in eigenvarieties.
\newblock {\em Math. Res. Lett.}, 18(5):983--990, 2011.

\bibitem[Maz89]{mazur}
B.~Mazur.
\newblock Deforming {G}alois representations.
\newblock In {\em Galois groups over {${\bf Q}$} ({B}erkeley, {CA}, 1987)},
  volume~16 of {\em Math. Sci. Res. Inst. Publ.}, pages 385--437. Springer, New
  York, 1989.

\bibitem[MB90]{mb}
Laurent Moret-Bailly.
\newblock Extensions de corps globaux \`a ramification et groupe de {G}alois
  donn\'{e}s.
\newblock {\em C. R. Acad. Sci. Paris S\'{e}r. I Math.}, 311(6):273--276, 1990.

\bibitem[Nek93]{nek}
Jan Nekov\'{a}\v{r}.
\newblock On {$p$}-adic height pairings.
\newblock In {\em S\'{e}minaire de {T}h\'{e}orie des {N}ombres, {P}aris,
  1990--91}, volume 108 of {\em Progr. Math.}, pages 127--202. Birkh\"{a}user
  Boston, Boston, MA, 1993.

\bibitem[Pat16]{pat}
Stefan Patrikis.
\newblock Deformations of {G}alois representations and exceptional monodromy.
\newblock {\em Invent. Math.}, 205(2):269--336, 2016.

\bibitem[Pat17a]{pat2}
Stefan Patrikis.
\newblock Deformations of {G}alois representations and exceptional monodromy,
  {II}: raising the level.
\newblock {\em Math. Ann.}, 368(3-4):1465--1491, 2017.

\bibitem[Pat17b]{pat-ks}
Stefan Patrikis.
\newblock Generalized {K}uga-{S}atake theory and good reduction properties of
  {G}alois representations.
\newblock {\em Algebra Number Theory}, 11(10):2397--2423, 2017.

\bibitem[Ram99]{r-lift}
Ravi Ramakrishna.
\newblock Lifting {G}alois representations.
\newblock {\em Invent. Math.}, 138(3):537--562, 1999.

\bibitem[Ram02]{r-def}
Ravi Ramakrishna.
\newblock Deforming {G}alois representations and the conjectures of {S}erre and
  {F}ontaine-{M}azur.
\newblock {\em Ann. of Math. (2)}, 156(1):115--154, 2002.

\bibitem[Sch15]{scholze}
Peter Scholze.
\newblock On torsion in the cohomology of locally symmetric varieties.
\newblock {\em Ann. of Math. (2)}, 182(3):945--1066, 2015.

\bibitem[Ser87]{serre}
Jean-Pierre Serre.
\newblock Sur les repr\'{e}sentations modulaires de degr\'{e} {$2$} de {${\rm
  Gal}(\overline{\bf Q}/{\bf Q})$}.
\newblock {\em Duke Math. J.}, 54(1):179--230, 1987.

\bibitem[Ser19]{vlad}
Vlad Serban.
\newblock A finiteness result for $p$-adic families of bianchi modular forms,
  2019.

\bibitem[Tay88]{taylor}
Richard~Lawrence Taylor.
\newblock {\em On congruences between modular forms}.
\newblock ProQuest LLC, Ann Arbor, MI, 1988.
\newblock Thesis (Ph.D.)--Princeton University.

\bibitem[Tay03]{taylor-icos}
Richard Taylor.
\newblock On icosahedral {A}rtin representations. {II}.
\newblock {\em Amer. J. Math.}, 125(3):549--566, 2003.

\end{thebibliography}
\end{document}